\documentclass[a4paper,12pt]{article}

\usepackage[bbgreekl]{mathbbol}
\usepackage{mathrsfs}
\usepackage{graphicx}
\usepackage{amsmath}
\usepackage{amsfonts}
\usepackage{amssymb}
\usepackage{amsthm}
\usepackage{color}
\usepackage{cancel}
\usepackage{subfigure}
\usepackage{appendix}
\usepackage{bm}
\usepackage{comment}
\usepackage{epstopdf}
\usepackage{multirow}
\usepackage{colortbl}
\usepackage{tikz}
\usepackage{booktabs}
\usepackage{url}
\usepackage[colorlinks,linkcolor=blue,anchorcolor=green,citecolor=blue]{hyperref}

\newtheorem{theorem}{Theorem}[section]
\newtheorem{lemma}{Lemma}[section]
\newtheorem{remark}{Remark}[section]

\renewcommand{\theequation}{\arabic{section}.\arabic{equation}}
\renewcommand{\thetheorem}{\arabic{section}.\arabic{theorem}}
\renewcommand{\thelemma}{\arabic{section}.\arabic{lemma}}

\renewcommand{\thefigure}{\arabic{section}.\arabic{figure}}

\newcommand{\bvec}[1]{\mathbf{#1}}

\newcommand{\vk}{\bvec{k}}

\newcommand{\vn}{\bvec{n}}

\newcommand{\vp}{\bvec{p}}
\newcommand{\vq}{\bvec{q}}
\newcommand{\vr}{\bvec{r}}

\newcommand{\vw}{\bvec{w}}

\newcommand{\vG}{\bvec{G}}

\newcommand{\vR}{\bvec{R}}

\newcommand{\Z}{\mathbb{Z}}
\newcommand{\R}{\mathbb{R}}

\newcommand{\node}{\mathrm{\Xi}}

\newcommand{\Node}{\bm{\node}}

\newcommand{\Ne}{N_{\rm e}}
\newcommand{\dd}{{\rm d}}

\title{A Hybrid Discontinuous Galerkin Method with Isogeometric-Planewaves Coupling for Periodic Full-Potential Electronic Structure Calculations
}
\author{Baowei Lai \thanks{{\it baoweilai@mail.bnu.edu.cn}.
Faculty of Arts and Sciences, Beijing Normal University, Zhuhai 519087, Guangdong, China.}, ~Xucheng Meng\thanks{{\it xcmeng@bnu.edu.cn}.
Faculty of Arts and Sciences, Beijing Normal University, Zhuhai 519087, Guangdong, China \& Guangdong Provincial Key Laboratory of Interdisciplinary Research and Application for Data Science, 
Beijing Normal-Hong Kong Baptist University, Zhuhai 519087, Guangdong, China.} ~ and
~Xiaoxu Li\thanks{Corresponding. {\it xiaoxuli@bnu.edu.cn}.
Faculty of Arts and Sciences, Beijing Normal University, Zhuhai 519087, Guangdong, China.}
}
\date{}

\begin{document}
\maketitle

\begin{abstract}

Full-potential electronic structure calculations for periodic systems retain the Coulomb singularity at the nuclei, which induces cusp behavior of the orbitals near the nuclei while leaving the interstitial region smooth. 
This multiscale regularity motivates a hybrid discretization framework that combines localized real-space discretizations in atomic patches with plane waves in the interstitial region. 
In this work, we employ tensor-product B-splines in the atomic patches and couple the two approximation spaces through a symmetric interior penalty discontinuous Galerkin formulation, yielding the isogeometric-plane wave (IGA-PW) method. 
To efficiently evaluate the restricted plane wave integrals, we develop a combined fast Fourier transform and Chebyshev correction strategy. 
We also construct a trace-block DG preconditioner to alleviate the conditioning deterioration caused by the SIPG penalty terms and accelerate the iterative eigensolver. 
In addition, we establish {\it a priori} error estimates for the corresponding linear eigenvalue problem, showing algebraic convergence near the nuclei and superalgebraic convergence in the interstitial region. 
Numerical experiments demonstrate the accuracy and efficiency of the proposed method.
\end{abstract}

\vspace{1cm}

\section{Introduction}\label{sec-intro} 
\setcounter{equation}{0}

Electronic structure calculations describe the energies and distributions of electrons, and play a fundamental role in materials science, quantum chemistry, solid-state physics, and condensed matter physics \cite{lebris03,martin05,saad2010numerical}. Among the various electronic structure models, Kohn-Sham density functional theory (DFT) \cite{hohenberg1964inhomogeneous,kohn1965self} is the most widely used framework for ground-state simulations, due to its effective balance between accuracy and computational cost \cite{lin2019numerical}.
For an $\Ne$-electron system with nuclei of charges $Z_k$ located at $\vR_k\in\mathbb{R}^3$ ($k=1,\ldots,M$), 
the Euler-Lagrange equations of the Kohn-Sham energy functional take the form
\begin{eqnarray}\label{eq:eigen}
H[\rho]\phi_i=\lambda_i\phi_i,
\qquad
\lambda_1\leq\lambda_2\leq\cdots\leq\lambda_{\Ne},
\end{eqnarray}
where the Kohn-Sham Hamiltonian is given by
\begin{eqnarray*}
H[\rho] = -\frac{1}{2}\Delta + V_{\rm ext} + V_{\rm H}[\rho] + V_{\rm xc}[\rho].
\end{eqnarray*}
Here $\displaystyle V_{\rm ext}({\vr})=-\sum_{k=1}^M\frac{Z_k}{|{\vr}-{\bf R}_k|}$ denotes the external potential generated by nuclear attraction,
$\displaystyle V_{\rm H} [\rho](\vr)=\int_{\mathbb{R}^3}\frac{\rho({\vr'})}{|{\vr}-{\vr'}|}{\dd \vr'}$ and
$ V_{\rm xc}[\rho]$ are the Hartree potential and exchange-correlation potential, respectively. 
The electron density $\displaystyle \rho(\vr)=\sum_{i=1}^{\Ne}|\phi_i({\vr})|^2$ determines $H[\rho_{\Phi}]$ and makes \eqref{eq:eigen} nonlinear.
It is usually solved by the self-consistent field (SCF) method, which repeatedly constructs the Hamiltonian from the current density, solves the resulting linear eigenvalue problem, and updates the density until convergence \cite{saad2010numerical}. In the periodic setting considered in this work, the problem is posed on a unit cell \(\Omega\), where the nuclei are periodically repeated over the underlying Bravais lattice and the orbitals satisfy periodic boundary conditions.

The nuclear Coulomb attraction induces cusp singularities in the orbitals near the nuclei. 
Two distinct approaches are commonly adopted in practice.
The {\it pseudopotential} approximation replaces the singular nuclear attraction and core-electron effects by a smooth effective potential \cite{martin05}, allowing plane wave discretizations to achieve superalgebraic convergence \cite{cances2012planewave}. 
However, the modeling error and transferability of pseudopotentials are difficult to quantify {\it a priori} \cite{cances2016optimal}.
Full-potential calculations, in contrast, retain the singular Coulomb attraction at the nuclei. The resulting orbitals exhibit a multiscale regularity structure, with low regularity near the nuclei and much higher regularity in the interstitial region \cite{flad08,fournais2002smooth}. Consequently, global plane wave discretizations lose their superalgebraic convergence in full-potential calculations, since the nuclear cusps destroy the global smoothness of the orbitals. 
Such a multiscale structure poses a fundamental challenge for spatial discretization: a uniform global basis must resolve both the singular near-nucleus behavior and the smooth interstitial region at the same approximation scale, leading to substantial over-resolution away from the nuclei.

Several discretization strategies have been developed for full-potential calculations, differing mainly in how they resolve the singular near-nucleus behavior and the smooth interstitial region. Methods based on linear combinations of atomic orbitals (Slater-type, Gaussian-type, and numerical atomic orbitals \cite{herring40,martin05}) efficiently capture the local singular behavior near the nuclei with relatively few degrees of freedom, but generally lack systematic convergence \cite{batcho2000computational}. The augmented plane wave (APW) and linearized augmented plane wave (LAPW) methods \cite{andersen75,sjostedt00,slater37} combine plane waves in the interstitial region with radial basis functions inside atomic spheres, and their approximation properties have been analyzed in \cite{chen2015apw}. However, the muffin-tin decomposition introduces additional interface matching conditions between different representations. Finite element methods \cite{bao2012adaptive,chen2014adaptive,motamarri2013higher,pask2001finite} provide systematic error control, but efficient resolution of the nuclear cusp typically requires local \(h\)- or \(p\)-adaptivity. These considerations naturally motivate hybrid discretizations that combine local resolution near the nuclei with high-order accuracy in the smooth interstitial region.

Isogeometric analysis (IGA), introduced by Hughes et al.\ \cite{hughes2005isogeometric}, provides high-order spline approximations with enhanced continuity and has been successfully applied to Kohn-Sham calculations \cite{cimrman2018convergence,cimrman2018isogeometric,masud2012bsplines,wang2025hierarchical}. 
However, single-patch IGA discretizations are still affected by the nuclear cusp singularity in full-potential calculations. Discontinuous Galerkin (DG) methods provide a natural and flexible framework for coupling different approximation spaces across subdomains, particularly when different regularity regimes require different local discretizations \cite{antonietti06}.

Within this framework, several multipatch DG-IGA methods \cite{CHAN201822,duvigneau2018isogeometric,moore2019space,proserpio2020framework} have been developed for spline coupling with varying polynomial degrees, mesh sizes, and local refinements. DG coupling has also been applied to adaptive local basis discretizations in electronic structure calculations \cite{lin2012adaptive,zhang2017adaptive}. For full-potential problems, \cite{li2025dgiga} developed a multi-patch DG-IGA framework with rigorous {\it a priori} eigenvalue error estimates. 
Related hybrid discretizations were also considered in \cite{li2019dg}, where radial basis functions inside atomic spheres were coupled with plane waves in the interstitial region.

In this paper, we present an Isogeometric-Plane Wave (IGA-PW) discretization for full-potential electronic structure calculations under periodic boundary conditions. 
The method employs tensor-product B-splines in localized atomic patches and plane waves in the interstitial region, with the two approximations coupled through a symmetric interior penalty DG formulation. 
This construction combines the local resolution capability of spline discretizations near the nuclei with the high-order efficiency of plane waves in the smooth interstitial region.

The main contributions of this work are threefold. 
First, the proposed framework replaces conventional atom-centered spherical decompositions by flexible atomic patches and naturally accommodates more general interior region decompositions and complex local geometries. 
Moreover, the framework is not tied to any particular local basis, allowing the use of spline, finite element, spectral element, and other localized real-space discretizations within the same DG coupling framework. 
Second, restricting plane waves to the interstitial region introduces discontinuous characteristic functions into the plane wave integrals, which deteriorates the convergence of direct FFT evaluation. To address this issue, we develop an FFT-Chebyshev integration strategy, where the smooth full-cell contribution is evaluated by FFT and the local atomic-patch contribution is removed through a Chebyshev correction. 
Third, since large penalty parameters in SIPG discretizations often lead to severe conditioning deterioration \cite{arnold2002unified,castillo2002performance}, we construct a trace-block DG preconditioner to accelerate the iterative eigensolver.
In addition, we establish an {\it a priori} error estimate for the IGA-PW approximation of full-potential linear eigenvalue problems, and numerical experiments demonstrate the accuracy and computational efficiency of the proposed method.

The remainder of this paper is organized as follows. 
In Section \ref{sec-problem}, we introduce the plane wave and B-spline approximation spaces and construct the IGA-PW discretization. 
Section \ref{sec: Error estimate} presents the main {\it a priori} error estimates, while the technical proofs are deferred to \ref{append:error_analysis}. 
In Section \ref{sec-implementation}, we discuss the efficient implementation of the proposed method, including the FFT-Chebyshev correction and the trace-block DG preconditioner. 
Numerical experiments are presented in Section \ref{sec:numerical-results} to demonstrate the accuracy and efficiency of the proposed method. 
Finally, concluding remarks are given in the last section.



\section{IGA-PW DG discretization}
\label{sec-problem} 
\setcounter{equation}{0}

In this section, we develop the IGA-PW discretization for periodic full-potential electronic structure calculations on the unit cell $\Omega=[-L/2,\,L/2]^d$ $(L>0,~ d=2,3)$ equipped with periodic boundary conditions. 
We first introduce the plane wave space and the B-spline space, together with their corresponding approximation estimates. 
Based on these ingredients, we then construct a symmetric interior penalty DG formulation, where the B-spline and the plane wave discretization are used in the distinct subregions.

\subsection{Plane wave approximation}

Let $\mathcal{R}= L\mathbb{Z}^d$ be a discrete periodic lattice and $\mathcal{R}^*=\frac{2\pi}{L}\Z^d$ be the dual lattice.
For $\vk\in\mathcal{R}^*$, we denote by $\displaystyle e_{\vk}(\vr)=|\Omega|^{-1/2} e^{\mathrm{i} \vk\cdot \vr}$ the plane wave
with the wavevector ${\bf k}$. 
The family $\{e_{\bf k}\}_{{\bf k}\in\mathcal{R}^*}$ forms an orthonormal basis set of
\begin{eqnarray*}
L_{\#}^2(\Omega)=\{u\in L^2_{\rm loc}(\R^d)~:~u~{\rm is}~\mathcal{R}{\rm -periodic}\}.
\end{eqnarray*}
For any $u\in L_{\#}^2(\Omega)$, we have
\begin{eqnarray*}
u(\vr)=\sum_{\vk\in\mathcal{R}^*} \hat{u}_{\vk} e_{\vk}(\vr) \quad \text{with} \quad \hat{u}_{\vk}
= (u, e_{\vk})_{L_{\#}^2(\Omega)} = |\Omega|^{-1/2} \int_{\Omega} u(\vr) e^{-\mathrm{i}{\vk\cdot \vr}} \dd\vr.
\end{eqnarray*}
We first define the Sobolev spaces of $\mathcal{R}$-periodic functions
\begin{eqnarray*}
H_{\#}^s(\Omega) = \left\{v=\sum_{\vk\in\mathcal{R}^*} \hat{v}_{\vk} e_{\vk}:~
\sum_{\vk\in\mathcal{R}^*} \left(1+|\vk|^2 \right)^s |\hat{v}_{\vk}|^2<\infty
\right\}
\end{eqnarray*}
with $s\in\mathbb{R}^+$.
For $K\in \mathbb{N}^+$, the associated finite-dimensional plane wave space is given by
\[
X_K
:=
\left\{
v\in L^2_{\#}(\Omega):~
v(\vr)=
\sum_{\substack{\vk\in\mathcal R^*\\ |\vk|\le 2\pi K/L}}
c_{\vk}e_{\vk}(\vr)
\right\}.
\]

For $v\in H^s_{\#}(\Omega)$, the orthogonal projection of $v$ onto $X_K$ takes the form
\begin{equation}
\Pi_K v=\sum_{\vk\in\mathcal{R}^*,|\vk|\leq \frac{2\pi}{L}K} \hat{v}_{\vk} e_{\vk},
\end{equation}
and the following estimate holds for any $0\leq t\leq s$,
\begin{equation}
\label{eq:add-rate-planewave}
\|v-\Pi_K v\|_{H^t_{\#}(\Omega)}
=
\min_{v_K\in X_K}\|v-v_K\|_{H^t_{\#}(\Omega)}
\leq C K^{t-s}\|v\|_{H^s_{\#}(\Omega)},
\end{equation}
where the constant $C$ is independent of $K$. 
It is shown that higher regularity of $v$ yields a faster plane wave approximation rate.

\subsection{B-spline approximation}

For positive integers $p$ and $n$, let  $\mathrm{\Xi} = \left \{0=\xi_1,\ldots,\xi_{n+p+1}=1\right\}$ be a non-decreasing sequence of real numbers in the parametric domain $[0,1]$.
Throughout this paper, we work with open knot vectors
\[
0=\xi_1=\cdots=\xi_{p+1}<\xi_{p+2}\le \cdots \le \xi_n
<\xi_{n+1}=\cdots=\xi_{n+p+1}=1.
\]

For a given $\mathrm{\Xi}$, the univariate B-spline basis functions of degree $p$ are defined by the Cox-de Boor recursion formula \cite{schumaker2007spline}
\begin{align}
B_{i,0}(\xi)  & = \left\{
\begin{array}{ll}\nonumber
1 \qquad \qquad &  \text{if} \quad \xi_{i} \leq \xi < \xi_{i+1}, \\[1ex]
0 & \text{otherwise}, \\
\end{array}\right. \\[1ex]
B_{i,p}(\xi)  & = \frac{\xi-\xi_{i}}{\xi_{i+p}
- \xi_{i}}B_{i,p-1}(\xi)
+\frac{\xi_{i+p+1}-\xi}{\xi_{i+p+1}-\xi_{i+1}}B_{i+1,p-1}(\xi) \qquad p\geq 1,
\end{align}
where any division by zero is defined to be zero. 
In general, a basis function of degree $p$ is $C^{p-m}$ across a knot with multiplicity $m$. For simplicity, we assume that all the internal knots have multiplicity one. Then the basis functions are globally $C^{p-1}$ on $(0, 1)$. 
In addition, each $B_{i,p}$ is non-negative and vanishes outside $(\xi_{i}, \xi_{i+p+1})$. We denote the corresponding
univariate spline space by
\[
S_p(\Xi):=\operatorname{span}\{B_{i,p}\}_{i=1}^n.
\]

In higher dimensions, the B-spline basis functions are constructed as tensor products of the univariate B-spline basis functions. 
Let \(\widehat{\Omega}:=[0,1]^d\) be the parametric domain, and let
\[
\Xi_\alpha=\{0=\xi_{1,\alpha},\ldots,\xi_{n_\alpha+p_\alpha+1,\alpha}=1\},
\qquad \alpha=1,\ldots,d,
\]
be open knot vectors in each coordinate direction. Set
\[
\mathbf p:=(p_1,\ldots,p_d),\qquad
\bm{\Xi}:=(\Xi_1,\ldots,\Xi_d),
\]
and define the multi-index set
\[
\mathcal I
:=
\left\{
\mathbf i=(i_1,\ldots,i_d):
1\le i_\alpha\le n_\alpha,\ \alpha=1,\ldots,d
\right\}.
\]
For each \(\mathbf i\in\mathcal I\), the tensor product B-spline basis function is defined by
\[
B_{\mathbf i,\mathbf p}(\bm\xi)
:=
\prod_{\alpha=1}^{d}
B^{(\alpha)}_{i_\alpha,p_\alpha}(\xi_\alpha),
\qquad
\bm\xi=(\xi_1,\ldots,\xi_d)\in\widehat{\Omega},
\]
where \(B^{(\alpha)}_{i_\alpha,p_\alpha}\) denotes the univariate B-spline basis function associated with the knot vector \(\Xi_\alpha\). The corresponding tensor product spline space is 
\[S_{\mathbf p}(\bm{\Xi}):=\operatorname{span}
\left\{B_{\mathbf i,\mathbf p}\right\}_{\mathbf i\in\mathcal I}.\] When \(p_\alpha=p\) for all \(\alpha=1,\ldots,d\), we write \(S_p(\bm{\Xi})\) and \(B_{\mathbf i,p}\) instead of \(S_{\mathbf p}(\bm{\Xi})\) and \(B_{\mathbf i,\mathbf p}\), respectively. Associated with the knot vectors $\Node$, the domain $[0,1]^d$ is partitioned into rectangles or rectangular prisms, forming a mesh $\mathcal{Q}_h$. For any $Q \in \mathcal{Q}_h$, we define $h_Q = \text{diam}(Q)$, and the global mesh size is given by $h = \max\{ h_Q, ~Q \in \mathcal{Q}_h \}$.

To transfer the spline construction from the parametric domain $\widehat{\Omega}$ 
to the physical patch $\Omega$, we introduce a bijective geometry map 
$G:\widehat{\Omega}\rightarrow\Omega$. The physical spline space  is then defined by
\[
V_h(\Omega)
:=
\{\widehat{v}\circ G^{-1}:\widehat{v}\in S_p(\bm{\Xi})\}.
\]
We use the standard tensor product spline quasi-interpolant $\widehat{\Pi}_h:L^2(\widehat{\Omega})\rightarrow S_p(\bm{\Xi})$ introduced in \cite{da2014mathematical}, and define the associated operator on the physical patch by
\[
\Pi_h v
:=
\bigl(\widehat{\Pi}_h(v\circ G)\bigr)\circ G^{-1},
\qquad
v\in L^2(\Omega).
\]
This operator satisfies the following approximation estimate \cite{schumaker2007spline}: for any $v\in H^t(\Omega)$ and $0\le r\le t\le p+1$,
\begin{equation}\label{eq:igaregularity}
\|v-\Pi_h v\|_{H^r(\Omega)}
\le
C h^{\,t-r}\|v\|_{H^t(\Omega)},  
\end{equation}
where the constant $C$ is independent of $h$.

\subsection{DG discretization}
\label{sec-discretization}

In this subsection, we construct DG discretization using the two 
approximation spaces introduced above. The computational domain 
is first decomposed into a region near the nuclei and an interstitial region. 
The wavefunction is then represented by B-splines in the region near the 
nuclei and by plane waves in the interstitial region, and the two 
representations are coupled by a symmetric interior penalty DG formulation.

As a model problem, we consider the Schr\"odinger-type eigenvalue problem: find $\lambda\in\mathbb{R}$ and $u\in H^1_{\#}(\Omega)$ such that $\|u\|_{L_{\#}^2(\Omega)}=1$ and
\begin{equation}
a(u,v)=\lambda(u,v) \qquad \forall v\in H^1_{\#}(\Omega),
\label{model-eq}
\end{equation}
where the bilinear form $a(\cdot,\cdot):H^1_{\#}(\Omega)\times H^1_{\#}(\Omega)\to\mathbb{C}$ is given by
\begin{equation}
a(u,v)=\frac12\int_\Omega \nabla u\cdot \nabla \overline v\,\dd\vr + \int_\Omega V u\overline v\,\dd\vr,
\label{bilinear-a}
\end{equation}
and $V\in L^2_{\#}(\Omega)$ is a real-valued $\mathcal R$-periodic potential.

To enable separate representations of the wavefunctions in different regions, the computational domain $\Omega$ is divided into atomic regions and an interstitial region (see Figure \ref{fig-division} (left) for the decomposition of a single-atom system, and Figure \ref{fig-division} (right) for a similar construction in a two-atom system).
Each atomic patch is a square neighborhood centered at the corresponding nucleus for $d=2$, or a cubic neighborhood for $d=3$. Unlike atom-centered spherical constructions, the present patch construction does not require precise alignment with the nuclear positions. It only requires that each nucleus remain inside the corresponding patch, providing greater flexibility for complex geometries and atomic configurations.

For simplicity, the construction is stated for a single atom located at the origin. A similar framework extends directly to multi-atom systems by taking the union of the corresponding atomic patches.
Throughout this paper, $\Omega_{\rm out}$ denotes the interstitial region and $\Omega_{\rm in}=[-R_{\rm in},R_{\rm in}]^d$
denotes the atomic region. The interface between the two regions is denoted by
$\Gamma:=\partial\Omega_{\rm in}$.
We assume that \(V\) is a real-valued \(\mathcal R\)-periodic potential with
isolated singularities at the nuclear positions.
Let \(\Sigma\) denote the set of all singular points, then
\[
V\in C_{\rm loc}^{\infty}(\mathbb R^d\setminus\Sigma)
\cap L^2_{\#}(\Omega)
\]
and $V$ has Coulomb-type singularities near the nuclei.
For a single atom at the origin, we have \(\Sigma=\mathcal R\).

\begin{figure}[ht]
\centering

\resizebox{0.8\linewidth}{!}{
    \begin{tikzpicture}
        \draw[black] (0,0) rectangle (6,4);
        \draw[black] (2,1) rectangle (4,3);
        \node at (3,2) {$\Omega_{\text{in}}$};
        \node at (0.5,0.5) {$\Omega_{\text{out}}$};
    \end{tikzpicture}
    \qquad \qquad \qquad \qquad
    \begin{tikzpicture}
        \draw[black] (0,0) rectangle (6,4);
        \draw[black] (1.5,1.5) rectangle (2.5,2.5);
        \draw[black] (3.5,1.5) rectangle (4.5,2.5);
        \node at (2,2) {$\Omega_{\text{in}}^1$};
        \node at (4,2) {$\Omega_{\text{in}}^2$};
        \node at (0.5,0.5) {$\Omega_{\text{out}}$};
    \end{tikzpicture}
}
\caption{The division of $\Omega$ into atomic regions $\Omega_{\rm in}$ and an interstitial region $\Omega_{\rm out}$.}
\label{fig-division}
\end{figure}
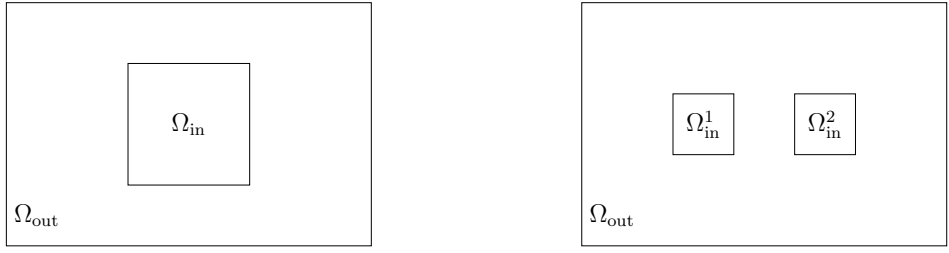

The regularity theory in \cite{fournais2002smooth,fournais04,fournais07} shows that the wavefunctions are smooth away from the nuclei, whereas the Coulomb singularities generate cusp behavior at the nuclear positions. 
As a result, the global regularity is limited by the nuclear singularities, so that a global plane wave discretization cannot achieve the high-order convergence rate in \eqref{eq:add-rate-planewave} for full-potential calculations.

Denote by $X_K(\Omega_{\rm out})$ the space of functions in $\Omega_{\rm out}$ expanded by plane waves
\begin{eqnarray*}
X_K(\Omega_{\rm out}) = \left\{v\in H^1(\Omega_{\rm out}):~v
= \sum_{|\vk|\leq \frac{2\pi}{L}K} c_{\vk} e_{\vk} \bigg|_{\Omega_{\rm out}} \right\}	
\end{eqnarray*}
and by $V_{h}(\Omega_{\rm in})$ the space of functions in $\Omega_{\rm in}$ generated by B-spline basis functions
\[
V_h(\Omega_{\rm in})
:=
\{\widehat{v}\circ G^{-1}:\widehat{v}\in S_p(\bm{\Xi})\},
\]
where $G:[0,1]^d\rightarrow\Omega_{\rm in}$ is the geometric mapping, which reduces to a simple affine rescaling in the current setting. 
Then the finite approximation space is defined by
\begin{align*}
\mathcal{S}^{K}_{h}(\Omega) :&=
X_K(\Omega_{\rm out})\oplus V_{h}(\Omega_{\rm in})
\\[1ex]
&= \left\{ v\in L^2_{\#}(\Omega):~v|_{\Omega_{\rm in}}\in V_{h}(\Omega_{\rm in}),~ v|_{\Omega_{\rm out}}\in X_K(\Omega_{\rm out}) \right\}.
\end{align*}

For vector-valued ${\bf w}$ and scalar-valued function $u$ which are not continuous on
the surface $\Gamma$, we introduce the jumps by
\begin{eqnarray*}
[{\bf w}]={\bf w}^+ \cdot {\bf n}^+ +{\bf w}^- \cdot {\bf n}^-,\quad\quad [u]=u^+{\bf n}^+
+u^-{\bf n}^-
\end{eqnarray*}
and the averages by
\begin{eqnarray*}
\{{\bf w}\}=\frac{1}{2}({\bf w}^+ +{\bf w}^-),\quad\quad \{u\}=\frac{1}{2}(u^++u^-),
\end{eqnarray*}
where $\vw^{\pm}$ and $u^{\pm}$ are traces of $\vw$ and $u$ on $\Gamma$ taken from inside and outside the atomic region, respectively, and $\vn^{\pm}$ are the corresponding outward normal unit vectors. Then the bilinear form  
$a^{\rm DG}(\cdot,\cdot): \mathcal{S}^{K}_{h} \times \mathcal{S}^{K}_{h}\rightarrow \mathbb C$
 is defined by
\begin{eqnarray}\label{eq-bilinear-DG}
\nonumber
a^{\rm DG}(u,v) &=& \int_{\Omega_{\rm in}}\left(\frac12\nabla u\cdot\nabla\overline v + V u\overline v\right)\,\dd\vr
+
\int_{\Omega_{\rm out}}\left(\frac12\nabla u\cdot\nabla\overline v + V u\overline v\right)\,\dd\vr\\[1ex]
&& -\frac12\int_\Gamma \{\nabla u\}\cdot [\overline v] \,\dd s
-\frac12\int_\Gamma \{\nabla \overline v\}\cdot [u] \,\dd s
+\sigma\int_\Gamma [u]\cdot [\overline v] \,\dd s,
\end{eqnarray}
where $\sigma=C_{\sigma}(K+h^{-1})$ is the discontinuity-penalization parameter with a sufficiently large constant $C_{\sigma}$ that is independent of $h$ and $K$.
Note that there are many other types of DG formulations (see, e.g., \cite{antonietti06,arnold2002unified}), 
and \eqref{eq-bilinear-DG} is the classical symmetric interior penalty Galerkin (SIPG) method \cite{antonietti06,arnold82,lin2012adaptive}.

Then we construct DG methods for the linear eigenvalue problem \eqref{model-eq}: find $\lambda^{\rm DG}
\in\mathbb{R}$ and $u^{\rm DG}\in \mathcal S_h^K(\Omega)$ such that 
$\|u^{\rm DG}\|_{L_{\#}^2(\Omega)} =1$ and
\begin{eqnarray}\label{eq-eigen-DG}
a^{\rm DG}(u^{\rm DG},v)=\lambda^{\rm DG}(u^{\rm DG},v)
\qquad\forall~v\in \mathcal S_h^K(\Omega).
\end{eqnarray}
We further define the broken Sobolev space
\begin{align*}
\widetilde{H}_{\#}(\Omega) = \{v\in L_\#^2(\Omega): v|_{\Omega_{\rm in}}\in H^1(\Omega_{\rm in}),\; v|_{\Omega_{\rm out}}\in H^1(\Omega_{\rm out})\} 
\end{align*}
equipped with the following DG-norm
\begin{eqnarray}\label{eq-norm-DG}
\|u\|_{\rm DG}^2
=
\|u\|_{H^1(\Omega_{\rm in})}^2
+
\|u\|_{H^1(\Omega_{\rm out})}^2
+
\sigma\|[u]\|_{L^2(\Gamma)}^2.
\end{eqnarray}

\section{Error estimates}
\label{sec: Error estimate}
\setcounter{equation}{0}

In this section, we establish an {\it a priori} error estimate for the IGA-PW approximation. 
The analysis builds on the DG-IGA framework of \cite{li2025dgiga}, while adapting it to the present hybrid setting, where local B-spline discretizations in the atomic patches are coupled with restricted periodic plane waves in the interstitial region. The two representations possess distinct approximation behaviors, and their interaction across the artificial interface must be incorporated into a unified eigenvalue analysis.

The resulting estimate naturally reflects the multiscale regularity of full-potential eigenfunctions: algebraic convergence is obtained near the nuclei, while spectral accuracy is retained in the smooth interstitial region. 
To the best of our knowledge, this is the first {\it a priori} eigenvalue error estimate for a DG formulation coupling local spline-based atomic discretizations with restricted periodic plane waves in full-potential calculations. 
Moreover, the analysis is not tied to the particular choice of B-spline basis functions, and the framework naturally extends to other localized basis discretizations as well as more general decompositions of the interior region. 
The main result is stated below, and the necessary proofs are given in \ref{append:error_analysis}.

For a nonzero eigenvalue $\lambda$ of \eqref{model-eq}, we denote by $M(\lambda)$ the associated eigenspace.
It was shown in \cite{fournais2002smooth,fournais04,fournais07} that eigenfunctions are smooth and, in fact, analytic away from the nuclei, while their regularity at the nuclei is limited by cusp behavior. 
Hence, the eigenspace satisfies that
\begin{equation}\label{eq:eigenspace}
M(\lambda) \subset H^{s_{\rm in}}(\Omega_{\rm in}) \oplus H^{\infty}(\Omega_{\rm out}).
\end{equation}

Throughout the sequel, the notation \(A\sim B\) means that \(A\) and \(B\) are of the same order up to constants independent of the discretization parameters, while \(A=O(B)\) means that \(A\) is bounded by \(B\) up to a constant independent of the discretization parameters.

\begin{theorem}[{\bf Error estimate of the eigenvalue problem}]
\label{them:convergence-rate}
Let $\lambda_i \neq 0$ be the $i$-th eigenvalue of \eqref{model-eq} and $(\lambda_i^{\rm DG}, u_i^{\rm DG})$ be the corresponding discrete eigenpair of \eqref{eq-eigen-DG}. 
Assume that $C_{\sigma}$ is sufficiently large and $h=O(K^{-1})$, then there exists $u_i\in M(\lambda_i)$ with $\|u_i\|_{L_{\#}^2(\Omega)}=1$ such that 
\begin{align}
\label{convergence-rate-eq}
\|u_i^{\rm DG}-u_i\|_{\rm DG}  & \leq C h^{k-1} \|u_i\|_{H^{k}(\Omega_{\rm in})} + C_s \left( K^{1-s} + \frac{K^{\frac{1}{2}-s}}{\sqrt{h}} \right), \qquad  \forall~s>1,
\end{align}
where $k=\min\{p+1,s_{\rm in}\}$.
Here $C$ and $C_s$ depend on the eigenvalue $\lambda_i$ but are independent of the discretization parameters $h$ and $K$, while $C_s$ may additionally depend on $s$ and $\|u_i\|_{H^s(\Omega_{\rm out})}$.
Moreover,
\begin{equation}
|\lambda_i^{\rm DG}-\lambda_i|
\le
C\|u_i^{\rm DG}-u_i\|_{\rm DG}^2.
\end{equation}
Furthermore, if $h\sim K^{-\alpha}$ with some $1\leq\alpha<3$, we have
\begin{equation}\label{eq:multiscale}
\|u_i^{\rm DG}-u_i\|_{L_{\#}^2(\Omega)} \leq C h^{\frac{3-\alpha}{2\alpha}} \|u_i^{\rm DG}-u_i\|_{\rm DG}.
\end{equation}
\end{theorem}

\begin{remark}[{\bf Superalgebraic convergence regarding $K$}]
The term
\[
C_s\left(K^{1-s}+\frac{K^{1/2-s}}{\sqrt h}\right)  \qquad  \forall~s>1
\]
indicates superalgebraic convergence with respect to the plane wave cutoff $K$. Although the convergence of $\frac{K^{1/2-s}}{\sqrt{h}}$ is slower than $K^{1-s}$ as $h=O(K^{-1})$, it still exhibits superalgebraic convergence behavior.
\end{remark}

\begin{remark}[{\bf Nonlinear eigenvalue problem}]
Within the framework of Kohn-Sham density functional theory, one needs to solve the nonlinear eigenvalue problem with the SCF iteration. 
Using our DG discretizations, the linear eigenvalue problem \eqref{eq-eigen-DG} is solved at each iteration step and mixing and stabilization schemes such as Roothaan, level-shifting, and DIIS algorithms (see, e.g., \cite{cances00,lebris03}) are used to accelerate or stabilize convergence.
	
If the exchange-correlation potential $V_{\rm xc}$ is sufficiently smooth and the trial states (from previous DG approximations) $\{\tilde{u}_1, \cdots, \tilde{u}_{\Ne}\}\in (\mathcal S_h^K)^{\Ne}$, then we have from arguments similar to those in \cite{flad08} that the eigenfunctions $\{u_i\}_{i=1}^{\Ne}$ of the Kohn-Sham Hamiltonian $H$ are smooth except at the positions of the nuclei.
This regularity together with the analysis in Theorem \ref{them:convergence-rate} gives the convergence rates for the DG approximations of the (linear) eigenvalue problem in each SCF iteration step.
	
Note that we have not provided an {\it a priori} error estimate for approximations of nonlinear eigenvalue problems, but only for linearized equations in SCF iterations.
We refer to \cite{cances2012planewave,chen13} for numerical analysis of nonlinear eigenvalue problems.
Nevertheless, the convergence behavior observed in our numerical experiments for the nonlinear problems is consistent with the theoretical results established for the corresponding linearized equations.
\end{remark}

\section{Efficient implementation}
\label{sec-implementation}
\setcounter{equation}{0}
\setcounter{figure}{0}

In this section, we discuss several implementation aspects of the proposed IGA-PW DG method. 
The hybrid discretization introduces three main computational challenges: the assembly of coupled DG matrices, the efficient evaluation of interstitial potential integrals, and the solution of the resulting ill-conditioned generalized eigenvalue problem. 

We first derive the explicit matrix representations associated with the plane wave block, the spline block, and the DG interface coupling terms. 
To accelerate the assembly of the interstitial potential matrices, we then develop an FFT-Chebyshev integration strategy that avoids the loss of efficiency caused by the discontinuous restriction to the interstitial region. 
Finally, we construct a trace-block DG preconditioner that exploits the sparsity structure induced by the DG penalty terms and significantly improves the convergence of iterative eigensolvers.

\subsection{Matrix assembly}
\label{subsec-integral}
Let \(\{e_{\vk_\vp}\}\) and \(\{\chi_{\vq}\}\) be the bases of \(X_K(\Omega_{\rm out})\) and
\(V_h(\Omega_{\rm in})\), respectively.  We write
\[
u^{\rm DG}
=
\sum_{\vp} (c_{\mathrm{out}})_{\vp}\,e_{\vk_\vp}
+
\sum_{\vq} (c_{\mathrm{in}})_{\vq}\,\chi_{\vq},
\]
and set
\[
c =
\begin{pmatrix}
c_{\mathrm{out}} \\
c_{\mathrm{in}}
\end{pmatrix}.
\]
Then the discrete eigenvalue problem takes the form
\begin{equation}
\label{eq-discrete-eig}
H c_i = \lambda_i^{\rm DG} M c_i, \qquad i=1,2,\cdots,N_{\rm e} .
\end{equation}
The Hamiltonian matrix $H$ and the overlap matrix $M$ have the block form
\[
H=
\begin{pmatrix}
H^a & H^c\\
(H^c)^* & H^b
\end{pmatrix},
\qquad
M=
\begin{pmatrix}
M^a & O\\
O & M^b
\end{pmatrix}.
\]
Here the superscripts $a$, $b$, and $c$ denote the plane wave block in 
$X_K(\Omega_{\rm out})$, the spline block in $V_h(\Omega_{\rm in})$, 
and the DG interface coupling block, respectively. The mixed overlap block vanishes
since the basis functions in \(X_K(\Omega_{\rm out})\) and \(V_h(\Omega_{\rm in})\)
have disjoint supports. 
We next present the explicit forms of the corresponding matrix entries for these three contributions.
\vskip 0.3cm
\noindent \textbf{(a)} For basis functions \(e_{\vk_\vp},e_{\vk_\vq}\in X_K(\Omega_{\rm out})\), we have
\begin{equation}
\label{eq-matrix-Ha}
H^a_{\vp\vq}
=
a^{\rm DG}\bigl(e_{\vk_\vq}|_{\Omega_{\rm out}},e_{\vk_\vp}|_{\Omega_{\rm out}}\bigr)
=
\frac{1}{2}\vk_\vp\cdot\vk_\vq\,U(\vk_\vq-\vk_\vp)
+V_{\rm out}(\vk_\vq-\vk_\vp)
+\mathfrak{D}^a_{\vp\vq},
\end{equation}
where
\[
U(\vk)
:=
\frac{1}{|\Omega|}\int_{\Omega_{\rm out}} e^{\mathrm{i}\vk\cdot\vr}\,\dd\vr
=
\delta_{\vk,\mathbf{0}}
-\frac{1}{|\Omega|}
\prod_{s=1}^d 2R_{\rm in}\,\operatorname{sinc}(k_sR_{\rm in}),
\]
and
\begin{equation}
\label{eq-Vout}
V_{\rm out}(\vk)
:=
\frac{1}{|\Omega|}
\int_{\Omega_{\rm out}}
V(\vr)e^{\mathrm{i}\vk\cdot\vr}\,\dd\vr .
\end{equation}
Here \(\vk=(k_1,\ldots,k_d)\), \(k_s\) denotes the \(s\)-th Cartesian component of
\(\vk\) and
\[
\operatorname{sinc}(t)=
\begin{cases}
\dfrac{\sin t}{t}, & t\neq 0,\\[1ex]
1, & t=0.
\end{cases}
\]
The DG contribution in \eqref{eq-matrix-Ha} is
\[
\mathfrak{D}^a_{\vp\vq}
=
\frac{\mathrm{i}}{4|\Omega|}
\int_{\Gamma}(\vk_\vq-\vk_\vp)\cdot \vn^+\,
e^{\mathrm{i}(\vk_\vq-\vk_\vp)\cdot\vr}\,\dd s
+
\frac{\sigma}{|\Omega|}
\int_{\Gamma} e^{\mathrm{i}(\vk_\vq-\vk_\vp)\cdot\vr}\,\dd s,
\]
where \(\vn^+\) is the outward unit normal vector of \(\Omega_{\rm in}\). Moreover,
\[
M^a_{\vp\vq}
=
\bigl(e_{\vk_\vq}|_{\Omega_{\rm out}},e_{\vk_\vp}|_{\Omega_{\rm out}}\bigr)
=
\frac{1}{|\Omega|}\int_{\Omega_{\rm out}}
e^{\mathrm{i}(\vk_\vq-\vk_\vp)\cdot\vr}\,\dd\vr
=
U(\vk_\vq-\vk_\vp).
\]

\noindent \textbf{(b)} For basis functions \(\chi_{\vp},\chi_{\vq}\in V_h(\Omega_{\rm in})\), we have
\[
H^b_{\vp\vq}
=
a^{\rm DG}(\chi_{\vq}|_{\Omega_{\rm in}},\chi_{\vp}|_{\Omega_{\rm in}})
=
\frac{1}{2}\int_{\Omega_{\rm in}}
\nabla\chi_{\vq}(\vr)\cdot\nabla\chi_{\vp}(\vr)\,\dd\vr
+
\int_{\Omega_{\rm in}} V(\vr)\chi_{\vq}(\vr)\chi_{\vp}(\vr)\,\dd\vr
+
\mathfrak{D}^b_{\vp\vq},
\]
where
\[
\label{dg-iga}
\mathfrak{D}^b_{\vp\vq}
=
-\frac{1}{4}\int_{\Gamma} (\nabla\chi_{\vq}\cdot\vn^+)\chi_{\vp}\,\dd s
-\frac{1}{4}\int_{\Gamma} (\nabla\chi_{\vp}\cdot\vn^+)\chi_{\vq}\,\dd s
+\sigma\int_{\Gamma}\chi_{\vq}\chi_{\vp}\,\dd s .
\]
Moreover,
\[
M^b_{\vp\vq}
=
\int_{\Omega_{\rm in}} \chi_{\vq}(\vr)\chi_{\vp}(\vr)\,\dd\vr .
\]

\noindent \textbf{(c)} For \(e_{\vk_\vp}\in X_K(\Omega_{\rm out})\) and \(\chi_{\vq}\in V_h(\Omega_{\rm in})\), we have
\begin{align*}
H^c_{\vp\vq}
&=
a^{\rm DG}\bigl(\chi_{\vq}|_{\Omega_{\rm in}},e_{\vk_\vp}|_{\Omega_{\rm out}}\bigr)\nonumber\\[1ex]
&=
\frac{1}{4|\Omega|^{1/2}}
\int_{\Gamma} (\nabla\chi_{\vq}\cdot\vn^+)
e^{-\mathrm{i}\vk_\vp\cdot\vr}\,\dd s
+
\frac{\mathrm{i}}{4|\Omega|^{1/2}}
\int_{\Gamma} (\vk_\vp\cdot\vn^+)
e^{-\mathrm{i}\vk_\vp\cdot\vr}\chi_{\vq}(\vr)\,\dd s \nonumber\\[1ex]
&\quad
-
\frac{\sigma}{|\Omega|^{1/2}}
\int_{\Gamma} e^{-\mathrm{i}\vk_\vp\cdot\vr}\chi_{\vq}(\vr)\,\dd s.
\end{align*}

\subsection{FFT-Chebyshev integration}
\label{subsec:pw-block-high-accuracy}
\numberwithin{table}{section}

Among all components in the matrix assembly, the most time-consuming part is the evaluation of the interstitial potential integral $V_{\rm out}(\vk)$ in \eqref{eq-Vout}. 
For a standard plane-wave discretization, such integrals can be efficiently computed by the fast Fourier transform (FFT). 
In the IGA-PW method, however, restricting the potential to $\Omega_{\rm out}$ introduces the discontinuous mask $\chi_{\Omega_{\rm out}}$. 
As a result, a direct FFT evaluation becomes significantly less efficient, since the jump discontinuity across $\Gamma$ slows down the convergence of the Fourier expansion and therefore requires substantially more Fourier modes to achieve the same accuracy.

In this subsection, we present an efficient evaluation strategy for $V_{\rm out}(\vk)$. For simplicity, the construction is described in two dimensions, while the three-dimensional case follows analogously.

To handle the discontinuity introduced by the restriction to $\Omega_{\rm out}$, we first replace \(V\) inside \(\Omega_{\rm in}\) by a smooth periodic extension \(\widetilde{V}\) that agrees with \(V\) on
\(\Omega_{\rm out}\) (see Figure \ref{fig:ex1-potential-smooth}). 
Let $0<b<a_c<R_{\rm in}$ and $g_0=V(b,0)$. Define
\begin{equation*}
\label{eq:smooth-extension-potential}
\widetilde{V}(\vr)
=
\left\{
\begin{array}{ll}
g_0, 
& r\le b,\\[1ex]
(1-\eta(r))V(\vr)+\eta(r)g_0, 
& b<r<a_c,\\[1ex]
V(\vr), 
& r\ge a_c,
\end{array}
\right.
\end{equation*}
where  \(r=|\vr|\) and
\begin{equation*}
\label{eq:smooth-cutoff-potential}
\eta(r)
=
1-\theta\!\left(\frac{r-b}{a_c-b}\right),
\qquad
\theta(t)
=
\frac{s(t)}{s(t)+s(1-t)},
\qquad
s(t)
=
\left\{
\begin{array}{ll}
e^{-1/t}, & t>0,\\[0.5ex]
0,        & t\le 0 .
\end{array}
\right.
\end{equation*}
Since the modification is confined to $\Omega_{\rm in}$, we have
$\widetilde V=V$ in $\Omega_{\rm out}$. 

\begin{figure}[htbp]
\centering
\includegraphics[width=0.6\textwidth]{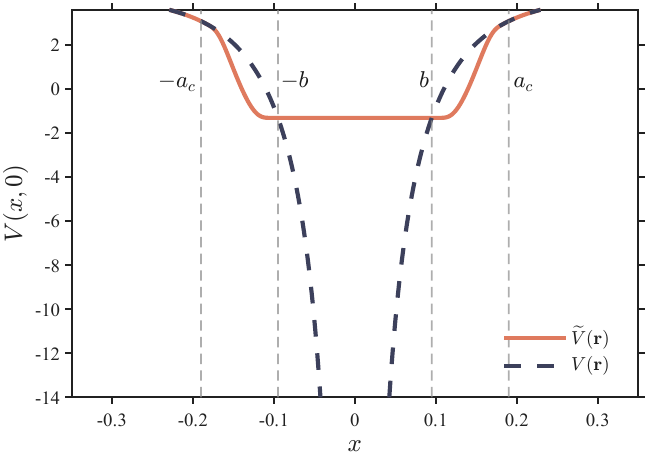}
\caption{The original potential $V(\vr)$ and its smoothed extension
$\widetilde{V}(\vr)$.}
\label{fig:ex1-potential-smooth}
\end{figure}

It is then straightforward to see that \eqref{eq-Vout} admits
the decomposition
\begin{equation}
\label{eq:full-minus-inner-decomposition}
V_{\rm out}(\vk)
=
\frac{1}{|\Omega|}
\int_{\Omega}\widetilde{V}(\vr)e^{\mathrm{i}\vk\cdot\vr}\,\dd\vr
-
\frac{1}{|\Omega|}
\int_{\Omega_{\rm in}}\widetilde{V}(\vr)e^{\mathrm{i}\vk\cdot\vr}\,\dd\vr
=: J_{\Omega}(\vk)-J_{\rm in}(\vk).
\end{equation}
The full-cell integral $J_{\Omega}$ can be evaluated by the FFT with spectral accuracy. 
For a given wave vector $\vk=(k_1,k_2)$, it remains to compute the inner-patch contribution in $\Omega_{\rm in}=[-R_{\rm in}, R_{\rm in}]^2$

\begin{equation}\label{eq:J_in}
    J_{\rm in}(\vk)
=
\frac{1}{|\Omega|}
\int_{-R_{\rm in}}^{R_{\rm in}}
\int_{-R_{\rm in}}^{R_{\rm in}}
\widetilde V(x,y)
e^{\mathrm{i}(k_1x+k_2y)}
\,\dd x\,\dd y .
\end{equation}
We approximate $\widetilde{V}$ in $\Omega_{\rm in}$ by a tensor product
Chebyshev expansion of the first kind
\begin{equation}
\label{eq:inner-chebyshev-expansion}
\widetilde{V}(x,y)
\approx
\sum_{r=0}^{n-1}\sum_{s=0}^{n-1}
c_{rs}\,
T_r\!\left(\frac{x}{R_{\rm in}}\right)
T_s\!\left(\frac{y}{R_{\rm in}}\right),
\end{equation}
where $T_r$ denotes the Chebyshev polynomial of degree $r$. Since
$\widetilde{V}$ is smooth on $\Omega_{\rm in}$, the approximation
\eqref{eq:inner-chebyshev-expansion} converges spectrally with respect to the polynomial degree.
To evaluate \eqref{eq:J_in}, we introduce the one-dimensional
Chebyshev integrals
\[
q_r(k)
:=
\int_{-R_{\rm in}}^{R_{\rm in}}
T_r\!\left(\frac{x}{R_{\rm in}}\right)
e^{\mathrm{i}k x}\,\dd x.
\]
Substituting \eqref{eq:inner-chebyshev-expansion}
into \eqref{eq:J_in} yields
\begin{equation}
\label{eq:inner-chebyshev-batched}
J_{\rm in}(k_{1},k_{2})
\approx
\frac{1}{|\Omega|}
\sum_{r=0}^{n-1}\sum_{s=0}^{n-1}
c_{rs}\,q_r(k_{1})\,q_s(k_{2})
=
\frac{1}{|\Omega|}
\left(Q^{T} C Q\right)_{k_1,k_2},
\end{equation}
where
\[
Q_{r\alpha}=q_r(k_\alpha),
\qquad
C_{rs}=c_{rs}.
\]
The matrix formulation
\eqref{eq:inner-chebyshev-batched} evaluates all wave vectors
simultaneously, thereby avoiding independent two-dimensional
quadratures for each mode and significantly reducing the assembly cost.

We demonstrate this efficiency in
Table~\ref{tab:pw-block-low-order-cheb-k25_m256}, which compares the accuracy and computational cost of evaluating
\eqref{eq-Vout} in the two- and three-dimensional examples in Section~\ref{sec:numerical-results}. 
Here, the masked FFT method refers to the direct FFT evaluation obtained by setting the grid values inside $\Omega_{\rm in}$ to zero on the global FFT grid. 
The results show that the proposed
FFT--Chebyshev integration strategy achieves substantially higher accuracy with significantly lower computational cost.

\begin{table}[htbp]
\centering
\caption{Comparison of the accuracy and computational time for evaluating \eqref{eq-Vout}.}
\label{tab:pw-block-low-order-cheb-k25_m256}
\begin{tabular}{llllll}
\toprule
Dimension & Method & FFT grid & $n$ & Time (s) & $\|V_{\rm out}
-V_{\rm out}^{\rm ref}\|_2$
 \\
\midrule
2D & Masked FFT & $32768^2$ & - & 457.39& $2.82\times 10^{-4}$ \\
2D & FFT-Chebyshev & $256^2$ & 48 & 0.92 & $3.70\times 10^{-5}$ \\
3D & Masked FFT & $2048^3$ & - & 792.72 & $4.34\times 10^{-3}$ \\
3D & FFT-Chebyshev & $300^3$& 80 & 2.45 & $2.23\times 10^{-6}$ \\
\bottomrule
\end{tabular}
\end{table}

\subsection{Trace-block DG preconditioning}
\label{subsec-precondition}
 
To solve the generalized eigenvalue problem \eqref{eq-discrete-eig}
by an iterative eigensolver, an effective preconditioner is essential.
We consider the shifted matrix
\[
A_{\tau}=H-\tau M,
\]
where $\tau$ is chosen close to the target eigenvalue.
As the mesh is refined, the DG penalty parameter $\sigma = C_\sigma(K+h^{-1})$ becomes large, and the resulting interface coupling severely deteriorates the conditioning of $A_{\tau}$, thereby slowing down the eigensolver.
A standard Jacobi (diagonal) preconditioner
\cite{wathen2015preconditioning}
fails to capture these interface couplings and is therefore insufficient.
To address this issue, we propose a trace-block DG preconditioner
(TB-DG), which treats the interface-coupled degrees of freedom
through an exact block solve while applying diagonal scaling to the
remaining interior degrees of freedom.

The main conditioning difficulty in \(A_\tau=H-\tau M\) comes from the DG penalty contribution, which we write explicitly as
\begin{equation*}
A_{\tau}
=
A_{\tau,0}
+\sigma P_{\Gamma},
\end{equation*}
where \(A_{\tau,0}\) collects all terms independent of the penalty parameter \(\sigma\).
Since most B-spline basis functions are locally supported away from the interface \(\Gamma\), the penalty matrix \(P_\Gamma\) is sparse and acts only on a small subset of degrees of freedom.

To exploit this structure, we partition the degrees of freedom into
\[
\gamma
:=
\Bigl\{
j:\sum_{k}|(P_{\Gamma})_{jk}|\neq 0
\Bigr\},
\qquad
\eta
:=
\{1,\ldots,N\}\setminus\gamma.
\]
Here \(\gamma\) contains all plane wave degrees of freedom together
with the B-spline basis functions having nonzero trace on
\(\Gamma\), while \(\eta\) contains the remaining interior
B-spline degrees of freedom. Reordering the indices accordingly gives
\[
P_{\Gamma}
=
\begin{pmatrix}
O & O\\
O & P_{\Gamma,\gamma\gamma}
\end{pmatrix},
\qquad
A_{\tau}
=
\begin{pmatrix}
A_{\tau,\eta\eta} & A_{\tau,\eta\gamma}\\
A_{\tau,\gamma\eta} & A_{\tau,\gamma\gamma}
\end{pmatrix}.
\]
In particular, the penalty contribution is entirely confined to the
\(\gamma\)-block, whereas \(A_{\tau,\eta\eta}\) remains independent
of \(\sigma\).

Motivated by this decomposition, we introduce the trace-block DG
preconditioner
\begin{equation}
\label{eq-precondition-ib}
P^{-1}
=
\begin{pmatrix}
D_{\eta}^{-1} & O\\
O & (A_{\tau,\gamma\gamma}+\delta I)^{-1}
\end{pmatrix},
\end{equation}
where $\displaystyle D_{\eta}
= |\operatorname{diag}(A_{\tau,\eta\eta})|$
denotes the absolute diagonal of the interior block.
The shift \(\delta\ge0\) is chosen to ensure the positive
definiteness of \(A_{\tau,\gamma\gamma}\).
The interior degrees of freedom are treated by inexpensive diagonal
scaling, while the interface-coupled block is handled through a sparse
exact or incomplete Cholesky factorization of
\(A_{\tau,\gamma\gamma}+\delta I\).

To assess the effectiveness of the proposed preconditioner,
Figure~\ref{fig:precondition} compares the unpreconditioned system, the Jacobi preconditioner, and the TB--DG preconditioner for Example~3 in Section~\ref{sec:numerical-results}.
Panels~(a) and~(b) show that the TB-DG preconditioner achieves the
smallest condition number and the shortest total solution time,
including the preconditioner setup cost.
Panel~(c) shows the convergence of the residual norm during the iterative eigensolve.
Without preconditioning or with the Jacobi preconditioner, the
residual decreases slowly and does not reach the tolerance
\(10^{-3}\) within \(1000\) iterations.
In contrast, the TB--DG preconditioner reduces the residual to
\(10^{-5}\) within about \(100\) iterations for \(p=1,2,3\), and the
convergence rate is only mildly affected by the spline order.

\begin{figure}[h!]
\subfigure[]{\includegraphics[width=0.32\textwidth]{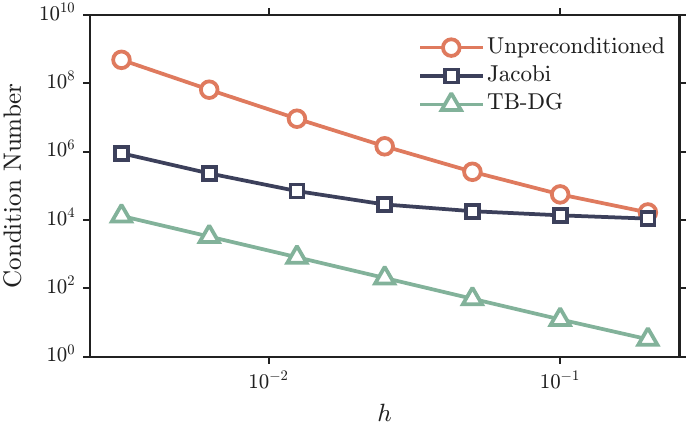}}
\subfigure[]{\includegraphics[width=0.32\textwidth]{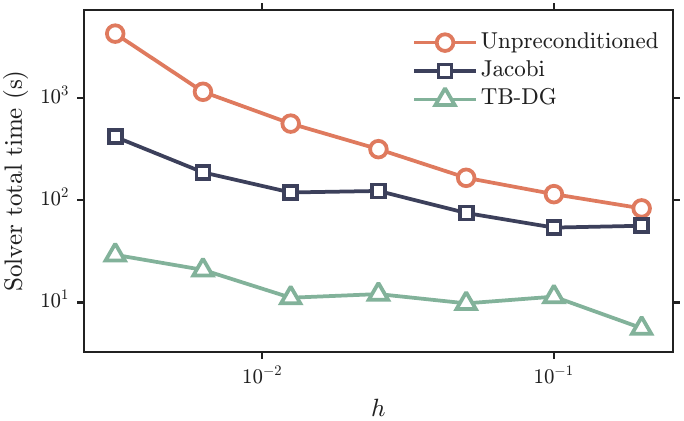}}
\subfigure[]{\includegraphics[width=0.33\textwidth]{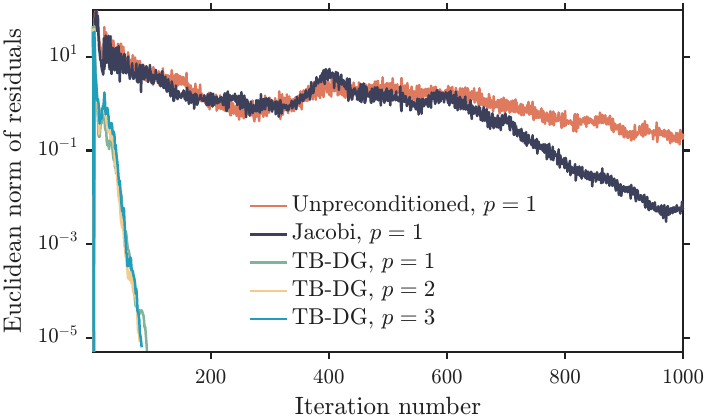}}
\caption{
Comparison of the unpreconditioned system, the Jacobi preconditioner, and the TB-DG preconditioner:
(a) condition numbers;
(b) total solution time, including both preconditioner setup and eigensolver time;
(c) convergence of the residual norm during the iterative eigensolve.
}
\label{fig:precondition}
\end{figure}

\section{Numerical experiments}
\label{sec:numerical-results}
\setcounter{equation}{0}
\setcounter{figure}{0}

In this section, we present numerical experiments for electronic structure problems. 
The experiments are designed to verify the convergence estimates in Theorem~\ref{them:convergence-rate} and to demonstrate the computational efficiency of the IGA-PW discretization on representative model problems.

The potential \(V(\vr)\) is chosen as the Ewald representation of the periodized Coulomb interaction, which decomposes the singular kernel into a short-range contribution and a smooth reciprocal-space summation. The two-dimensional Ewald formulation follows \cite{holzmann2005optimized}, while the three-dimensional counterpart is taken from~\cite{natoli1995optimized}. 
Throughout the experiments, the generalized eigenvalue problems are solved using the eigensolver PRIMME \cite{PRIMME2017algorithm}. 
Reference solutions are computed by the IGA--PW method on sufficiently fine discretizations. 
The implementation code is available in~\cite{DG_IGA_Eigenvalue_Codes}.

\vskip 0.5cm

\noindent{\bf Example 1 (2D linear problem for a single-atom system).}
We first consider the two-dimensional linear eigenvalue problem: find $\lambda\in\mathbb{R}$ and $u\in H^1_{\#}(\Omega)$ with $\|u\|_{L_\#^2(\Omega)}=1$ such that
\[
\left(-\frac{1}{2}\Delta +V \right)u = \lambda u,
\]
where $\Omega=[-2,2]^2$ and $\Omega_{\rm in}=[-0.2,0.2]^2$.
The Ewald potential takes the form
\begin{equation}
\label{eq:ex1-potential}
 V(\vr) = -\frac{\operatorname{erfc}(\alpha |\vr|)}{| \vr|}
 - \frac{2 \pi}{|\Omega|}\sum_{\substack{\vG\in\mathcal{R}^*\\ \vG \ne \bm 0}}
 \frac{1}{|\vG|}\operatorname{erfc}\!\left(\frac{|\vG|}{2\alpha}\right)e^{\mathrm{i} \vG \cdot \vr}
 + \frac{2\alpha}{\sqrt{\pi}}
\end{equation}
with $\alpha=5$.
In practice, only a few reciprocal-space modes with \(|\vG|\le2\) are retained in the summation, since the complementary error function \(\operatorname{erfc}\) decays exponentially fast.

We first show the degrees of freedom and the total time required for the plane wave, IGA, and IGA-PW discretizations in Figure~\ref{fig:ex1-compare}. At the same accuracy level, the IGA-PW method requires fewer degrees of freedom and less computational time than the plane wave and IGA methods, demonstrating the superior efficiency of the IGA-PW discretization.

\begin{figure}[h!]
\subfigure[]{\includegraphics[width=0.48\textwidth]{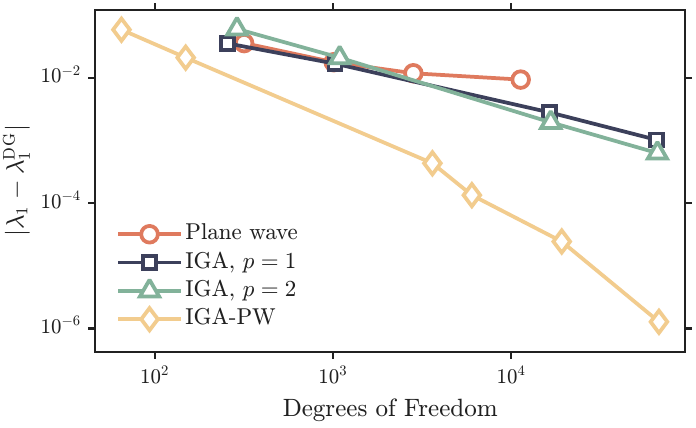}}
\subfigure[]{\includegraphics[width=0.48\textwidth]{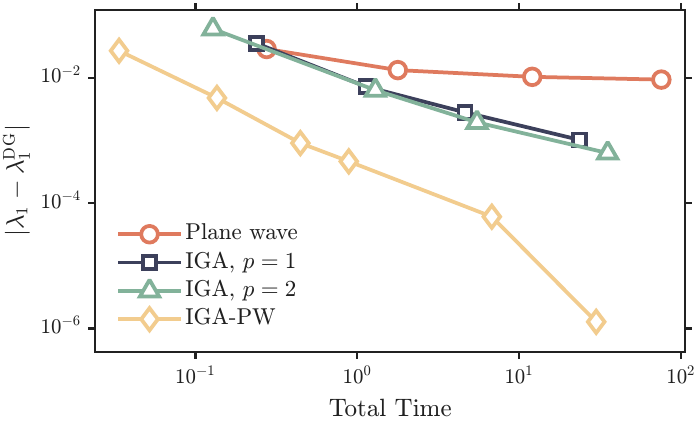}}
\caption{(Example 1) Comparison of the degrees of freedom and computational time for the plane wave, IGA, and IGA-PW methods.}
\label{fig:ex1-compare}
\end{figure}

We then examine the convergence with respect to the spline mesh size in
Figures~\ref{fig:ex1-eig-h} and \ref{fig:ex1-eigfunc-h}. The first two eigenvalue errors exhibit second-order convergence, which is limited by the cusp singularity of the corresponding eigenfunctions near the nucleus \cite{maday2019regularity}, see also Figure~\ref{fig:ex1-ref}. 
For the smoother eigenfunctions, higher-order convergence is recovered as the B-spline degree increases, and the eigenfunction errors measured in both \(L^2\) and DG norms exhibit the same behavior. In particular, the third and fourth eigenvalues form a degenerate pair, and the numerical results show that the proposed method remains stable and accurate in the presence of eigenvalue multiplicity. Overall, the observed convergence rates agree well with Theorem~\ref{them:convergence-rate}.

\begin{figure}[h!]
\subfigure[$p = 1$]{\includegraphics[width=0.48\textwidth]{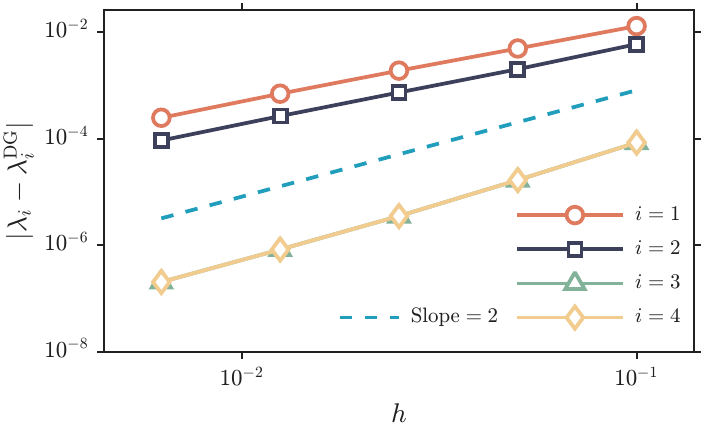}}
\subfigure[$p = 2$]{\includegraphics[width=0.48\textwidth]{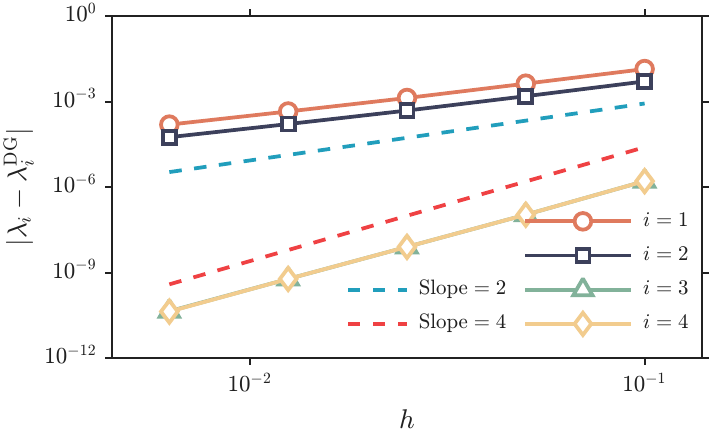}}
\caption{(Example 1) Convergence of the eigenvalues with respect to $h$ for $K=30$.}
\label{fig:ex1-eig-h}
\end{figure}

\begin{figure}[h!]
\subfigure[$p = 1$]{\includegraphics[width=0.48\textwidth]{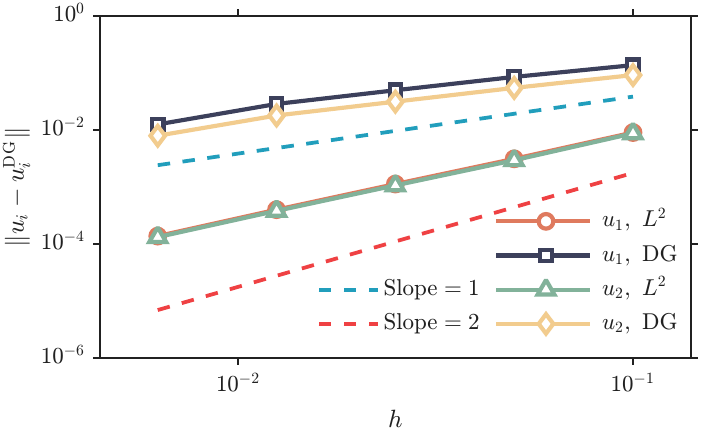}}
\subfigure[$p = 2$]{\includegraphics[width=0.48\textwidth]{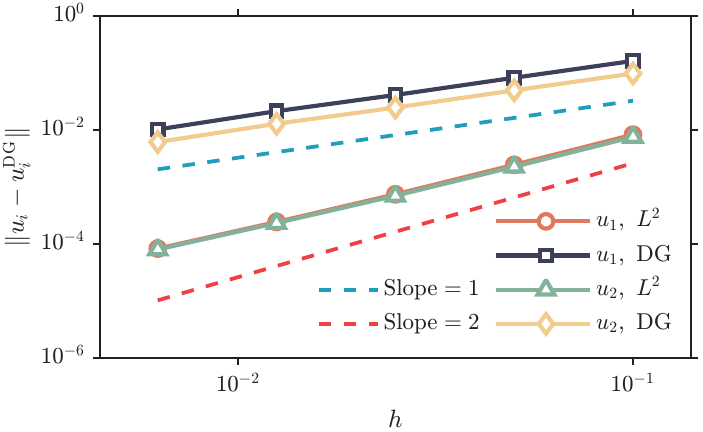}}
\caption{(Example 1) Convergence of the eigenfunctions in $L^2$ and DG norms with respect to $h$ for $K=30$.}
\label{fig:ex1-eigfunc-h}
\end{figure}

\begin{figure}[h!]
\subfigure[$u_1$]{\includegraphics[width=0.24\textwidth,height=0.24\textwidth]{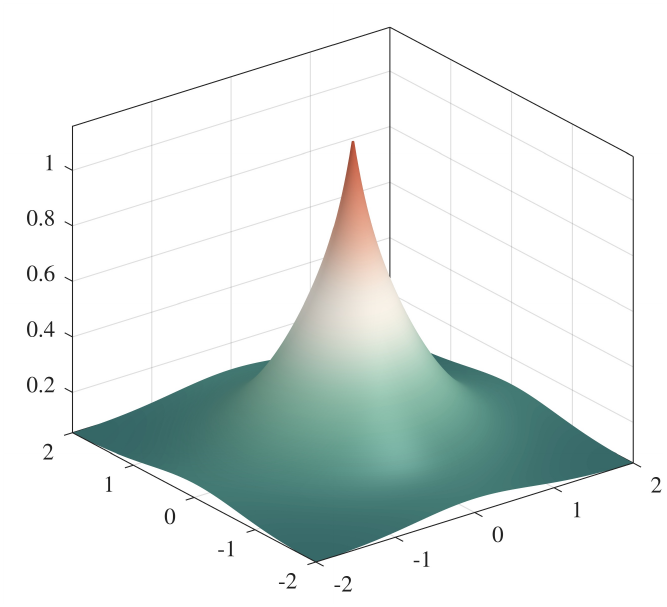}}
\subfigure[$u_2$]{\includegraphics[width=0.24\textwidth,height=0.24\textwidth]{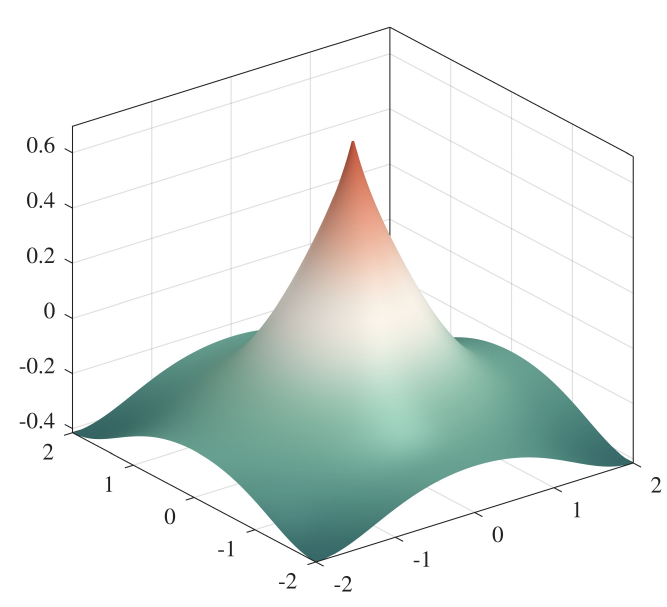}}
\subfigure[$u_3$]{\includegraphics[width=0.24\textwidth,height=0.24\textwidth]{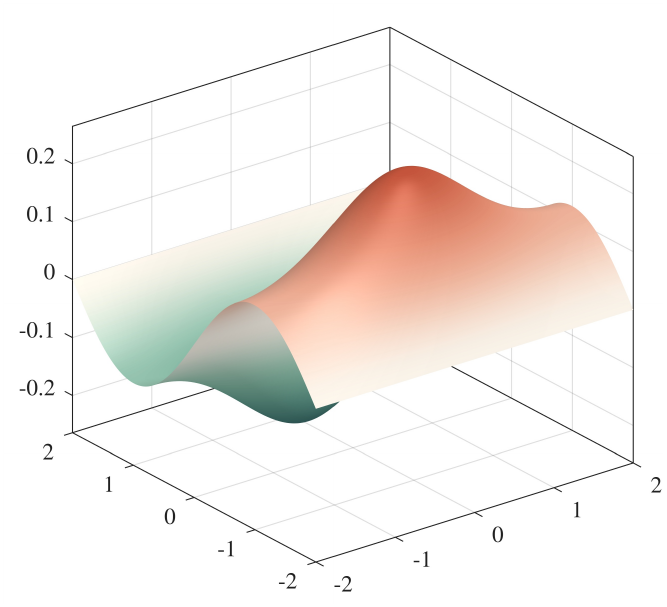}}
\subfigure[$u_4$]{\includegraphics[width=0.24\textwidth,height=0.24\textwidth]{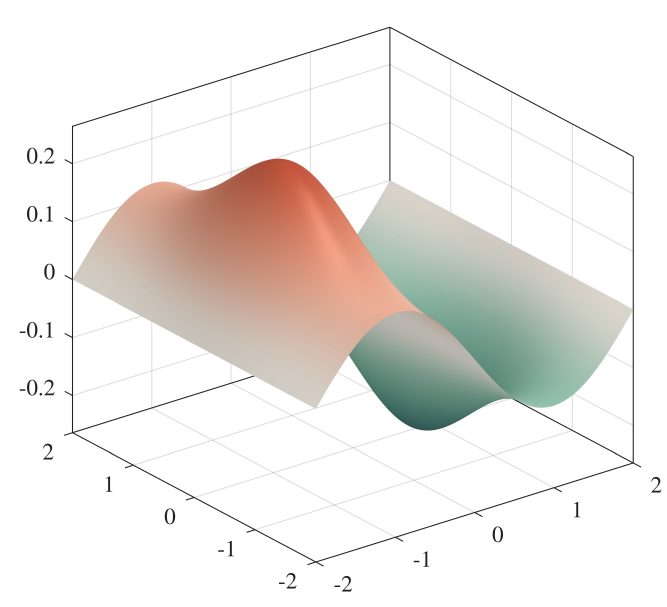}}
\caption{(Example 1) The reference eigenfunctions.}
\label{fig:ex1-ref}
\end{figure}

We next fix the spline discretization and vary the plane wave cutoff in Figure~\ref{fig:ex1-k}. It is observed that the errors decay rapidly as $K$ increases, which confirms the superalgebraic plane wave contribution in Theorem~\ref{them:convergence-rate}. 
Figure~\ref{fig:ex1-scaled-errors} validates the multiscale refinement estimate~\eqref{eq:multiscale}, where $h\sim K^{-\alpha}$ is adopted. It is observed that the scaled DG error $h^\gamma\|u_1-u_{1}^{\mathrm{DG}}\|_{\mathrm{DG}}$ exhibits the same decay behavior as the $L^2$ error.

Finally, Figure~\ref{fig:ex1-penalty} illustrates the effect of the penalty parameter \(C_\sigma\) and the TB-DG preconditioner. 
The penalty parameter \(C_\sigma\) in the SIPG bilinear form is required to be sufficiently large in order to ensure coercivity and stability \cite{arnold2002unified}.
Figure~\ref{fig:ex1-penalty}(a) shows that the eigenvalue error is essentially insensitive to \(C_\sigma\) once it exceeds a moderate threshold. Moreover, the curves obtained with and without the TB-DG preconditioner almost coincide, indicating that the preconditioner does not affect the discrete eigenvalue approximation. Figure~\ref{fig:ex1-penalty}(b) further shows that the TB-DG preconditioner significantly reduces the condition number over a wide range of \(C_\sigma\).

\begin{figure}[h!]
\subfigure[]{\includegraphics[width=0.48\textwidth]{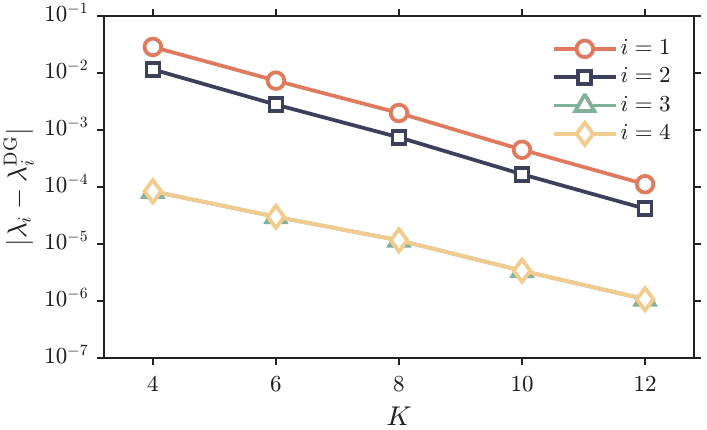}}
\subfigure[]{\includegraphics[width=0.48\textwidth]{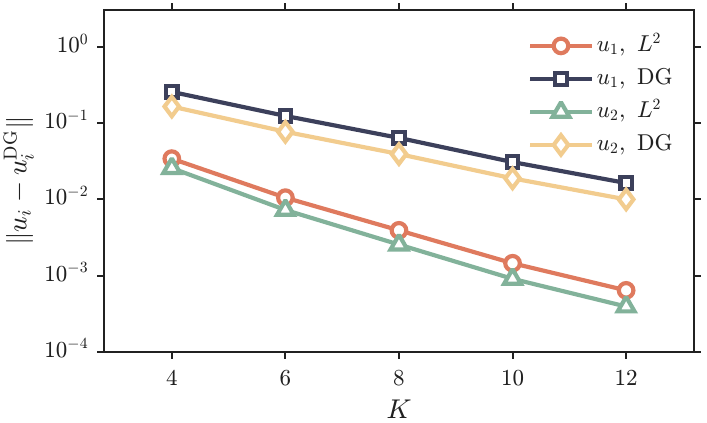}}
\caption{(Example 1) Convergence of the eigenvalues and the eigenfunctions in $L^2$ and DG norms with respect to the plane wave cutoff $K$, where $p=2$ and $h=0.4/2^7$.}
\label{fig:ex1-k}
\end{figure}

\begin{figure}[h!]
\subfigure[]{\includegraphics[width=0.48\textwidth]{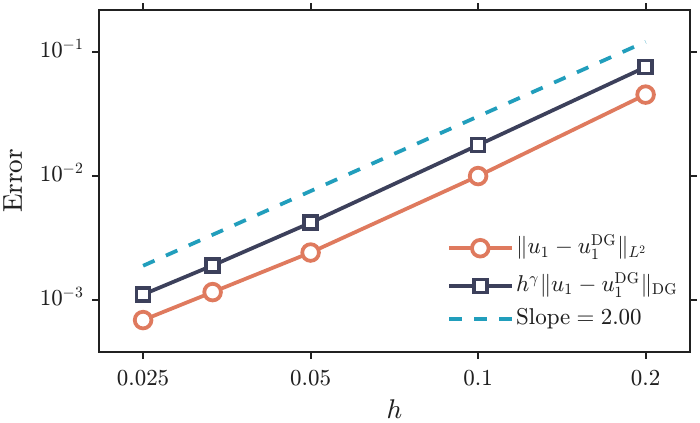}}
\subfigure[]{\includegraphics[width=0.48\textwidth]{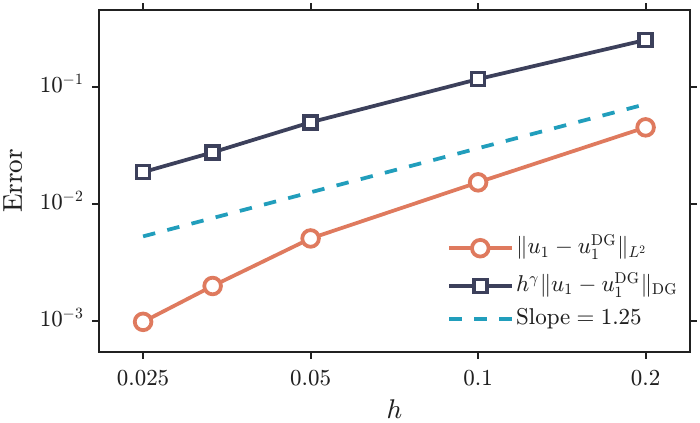}}
\caption{(Example 1) Convergence of $\|u_1-u_{1}^{\mathrm{DG}}\|_{L^2}$ and $h^\gamma\|u_1-u_1^{\mathrm{DG}}\|_{\mathrm{DG}}$ with respect to $h$, where $\gamma=(3-\alpha)/(2\alpha)$: (a) $\alpha=1.0$; (b) $\alpha=2.0$.}
\label{fig:ex1-scaled-errors}
\end{figure}

\begin{figure}[h!]
\subfigure[]{\includegraphics[width=0.48\textwidth]{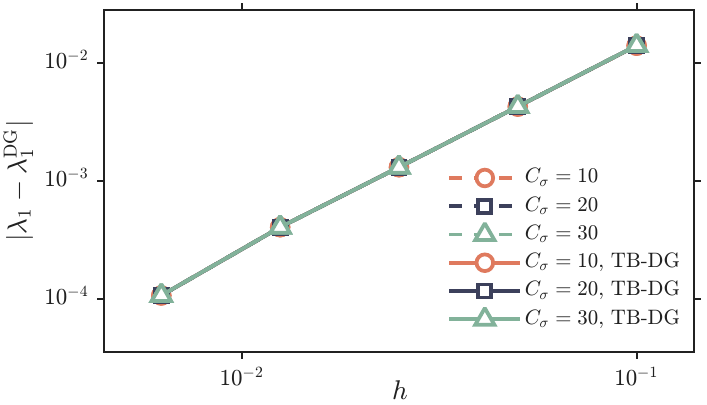}}
\subfigure[]{\includegraphics[width=0.48\textwidth]{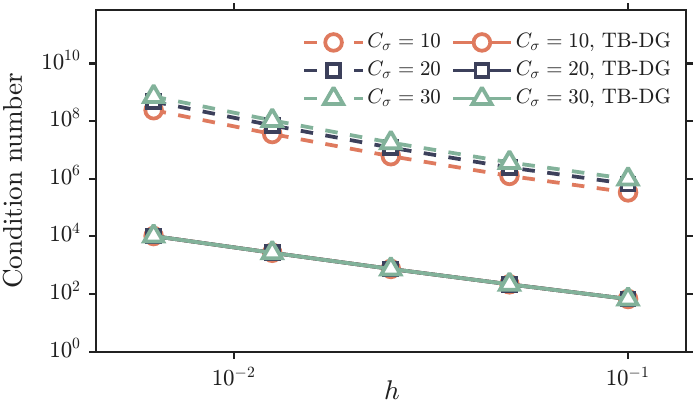}}
\caption{(Example 1) Effect of the penalty parameter $C_\sigma$ in the 
first eigenvalue error and condition number:
(a) convergence with and without the TB-DG 
preconditioner; 
(b) condition numbers of the unpreconditioned and TB-DG preconditioned 
systems.}
\label{fig:ex1-penalty}
\end{figure}

\vspace{0.5cm}

\noindent{\bf Example 2 (2D linear problem for a two-atom system).}
We next consider a two-atom system: find $\lambda\in\mathbb{R}$ and $u\in H^1_{\#}(\Omega)$ with $\|u\|_{L_\#^2(\Omega)}=1$ such that
\[
\left(-\frac{1}{2}\Delta +V_{1}+V_{2} \right)u = \lambda u,
\]
where $\Omega=[-2,2]^2$ and
\begin{equation*}
V_i(\vr) = -\frac{\operatorname{erfc}(\alpha |\vr - \vR_i|)}{|\vr - \vR_i|}
- \frac{2 \pi}{|\Omega|}\sum_{\substack{\vG\in\mathcal{R}^*\\ \vG \ne \bm 0}}
\frac{1}{|\vG|}\operatorname{erfc}\!\left(\frac{|\vG|}{2\alpha}\right)e^{\mathrm{i} \vG \cdot  (\vr - \vR_i)}
 + \frac{2\alpha}{\sqrt{\pi}}, \qquad i=1,2.
\end{equation*}
The atomic positions are $\vR_1=(-1,0)$ and $\vR_2=(1,0)$, and the inner 
domain is $\Omega_{\rm in}=[-1.2,-0.8]\times[-0.2,0.2]\cup
[0.8,1.2]\times[-0.2,0.2]$.

We first present the convergence with respect to the spline mesh size in
Figure~\ref{fig:ex2-eig-h}, and then display the convergence with respect to the
plane wave cutoff in Figure~\ref{fig:ex2-k}. It is observed that the errors exhibit algebraic convergence in $h$ and superalgebraic convergence in $K$, which agrees well with our theoretical results.
Figure~\ref{fig:ex2-eigenfunction-error} compares the first eigenfunctions and their absolute errors for the plane-wave, IGA, and IGA-PW discretizations. 
Although the three methods produce visually similar eigenfunctions, the error distributions indicate that the IGA-PW approximation is generally closer to the reference solution. 
In particular, the hybrid discretization provides a more accurate description of the nuclear cusp.

\begin{figure}[h!]
\subfigure[]{\includegraphics[width=0.48\textwidth]{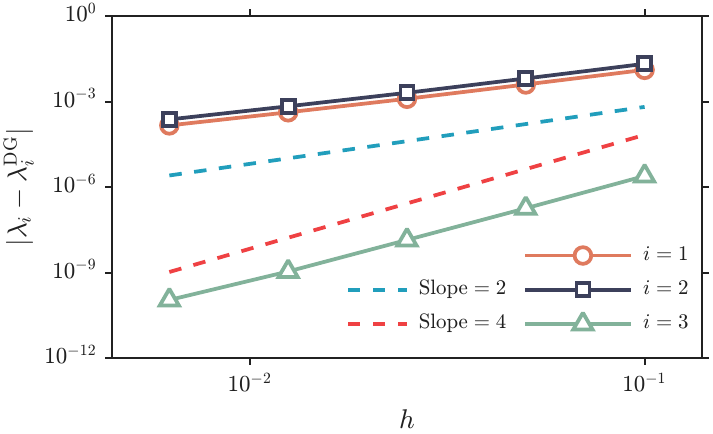}}
\subfigure[]{\includegraphics[width=0.48\textwidth]{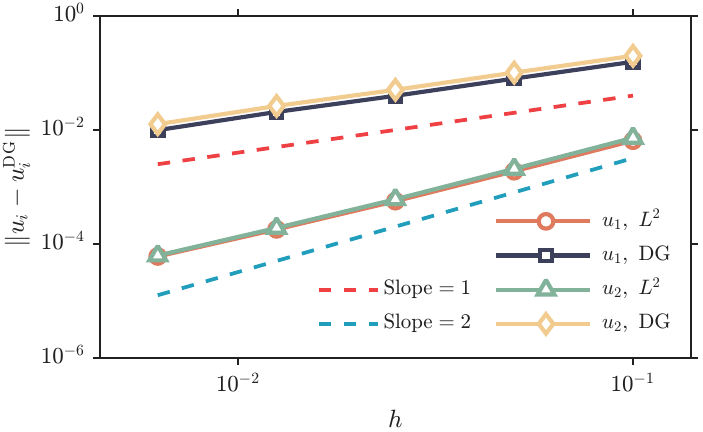}}
\caption{(Example 2) Convergence of the eigenvalues and eigenfunctions with respect to $h$, where $K=30$ and $p=2$.}
\label{fig:ex2-eig-h}
\end{figure}

\begin{figure}[h!]
\subfigure[]{\includegraphics[width=0.48\textwidth]{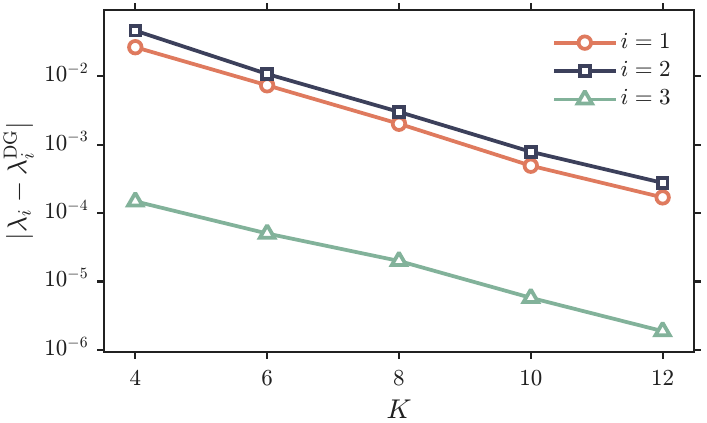}}
\subfigure[]{\includegraphics[width=0.48\textwidth]{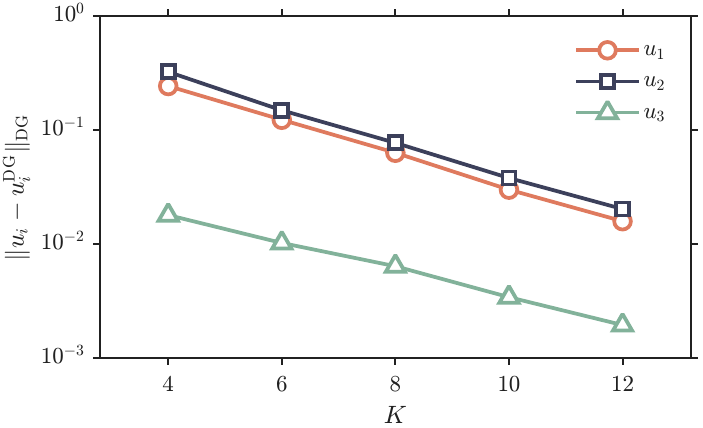}}
\caption{(Example 2) Convergence of the eigenvalues and the corresponding eigenfunctions in the DG norm with respect to the plane wave cutoff $K$, where $p=2$ and $h=0.4/2^7$.}
\label{fig:ex2-k}
\end{figure}

\begin{figure}[h!]
\centering
\setlength{\subfigcapskip}{10pt}

\subfigure[First eigenfunction obtained by plane wave, IGA, and IGA-PW.]{%
\begin{minipage}{0.98\textwidth}
\centering
\includegraphics[height=0.27\linewidth,keepaspectratio]{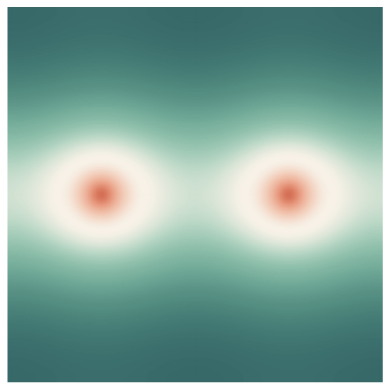}%
\hspace{0.05\linewidth}%
\includegraphics[height=0.27\linewidth,keepaspectratio]{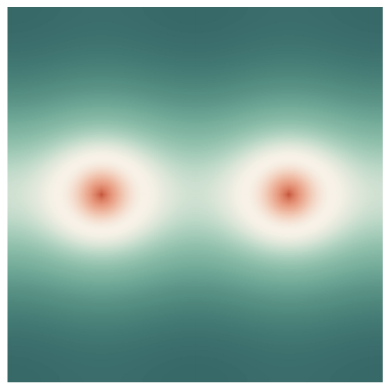}%
\hspace{0.05\linewidth}%
\includegraphics[height=0.27\linewidth,keepaspectratio]{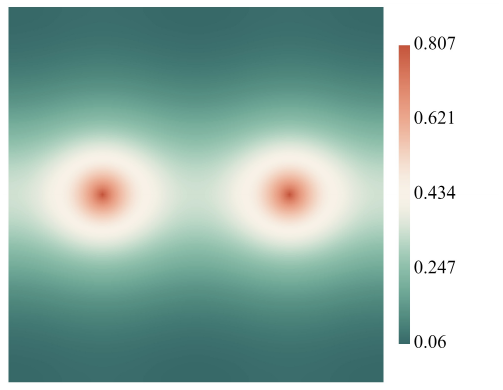}%
\end{minipage}}

\subfigure[Absolute error of the first eigenfunction for the plane wave, IGA, and 
IGA-PW discretizations, with $L^2$ errors $6.80\times10^{-3}$, $3.85\times10^{-3}$, 
and $1.61\times10^{-3}$, respectively.]{%
\begin{minipage}{0.98\textwidth}
\centering
\includegraphics[height=0.27\linewidth,keepaspectratio]{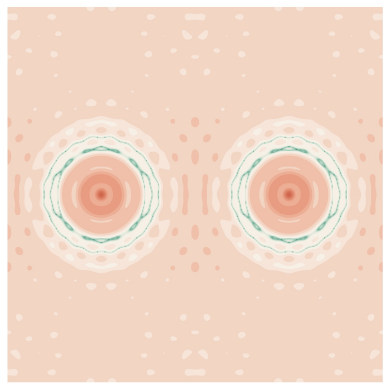}%
\hspace{0.05\linewidth}%
\includegraphics[height=0.27\linewidth,keepaspectratio]{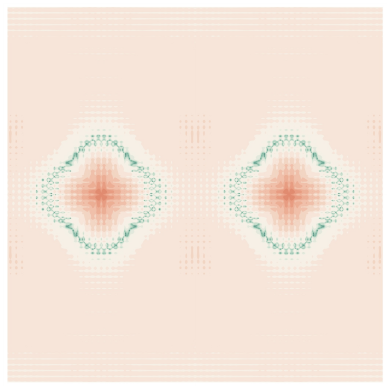}%
\hspace{0.05\linewidth}%
\includegraphics[height=0.27\linewidth,keepaspectratio]{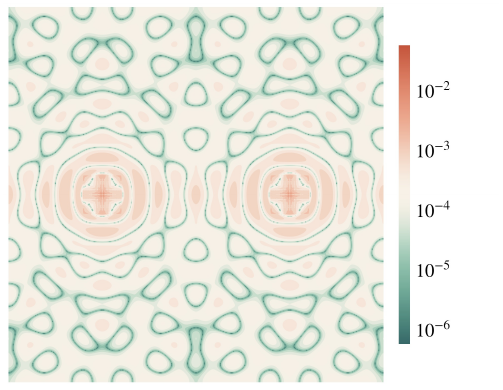}%
\end{minipage}}

\caption{(Example 2) The first eigenfunction and its absolute error for the plane wave, IGA, and IGA-PW discretizations, where the numbers of degrees of freedom are
$1257, 4096, 895$, respectively.}
\label{fig:ex2-eigenfunction-error}
\end{figure}

\vspace{0.5cm}

\noindent{\bf Example 3 (Modeling 2D Bose-Einstein condensates).}
Consider the following Gross-Pitaevskii equation arising in the modeling of Bose--Einstein condensates \cite{pitaevskii2003bose}: find $\lambda\in\mathbb{R}$ and $u\in H^1_{\#}(\Omega)$ with $\|u\|_{L_\#^2(\Omega)}=1$ such that
\[
\left(-\frac{1}{2}\Delta + V + \rho \right)u= \lambda u,
\]
where the electron density is defined as $\rho(\vr)=|u(\vr)|^2$. We use the same domain and external potential $V$ in \eqref{eq:ex1-potential}.

The nonlinear problem is usually solved by the self-consistent field (SCF) iteration. Given an input state $u^{(m)}_{\rm in}$ in the $m$-th iteration, we form the electron density $\rho^{(m)}_{\rm in}=|u^{(m)}_{\rm in}|^2$ and solve the linearized eigenvalue problem
\begin{equation*}
\left(-\frac{1}{2}\Delta+V+\rho^{(m)}_{\rm in}\right)u=\lambda u
\end{equation*}
to obtain an output eigenpair $(\lambda^{(m)}_{\rm out},u^{(m)}_{\rm out})$. 
The next input state is generated by a linear mixing step, $u^{(m+1)}_{\rm in}=\theta u^{(m)}_{\rm in}+(1-\theta)u^{(m)}_{\rm out}$ with $0<\theta<1$, and the iteration is repeated until the eigenvalue and density are converged.

Figure~\ref{fig:ex3-convergence} presents the convergence of the first eigenpair. 
The observed numerical results remain consistent with the theoretical predictions, although the analysis is established only for the corresponding linearized eigenvalue problem.
Figure~\ref{fig:ex3-error-surface} further illustrates the error distribution of the first eigenfunction. 
The scale of the \(z\)-axis shows that the IGA-PW method achieves smaller errors near the nuclei while preserving a smooth representation in the interstitial region.

\begin{figure}[h!]
  \subfigure[]
  {\includegraphics[width=0.32\textwidth]{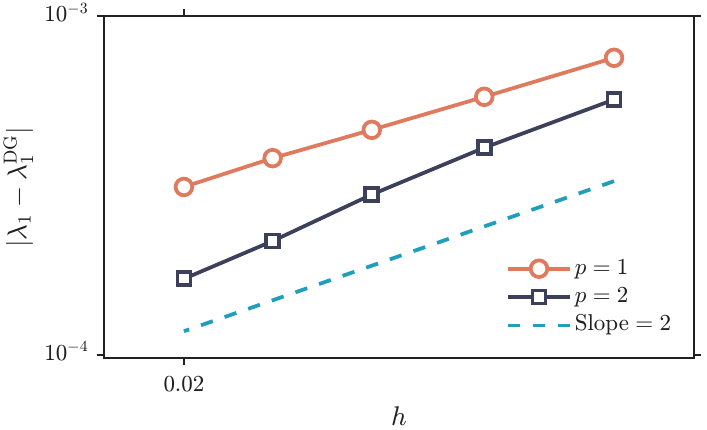}}
  \subfigure[]
  {\includegraphics[width=0.32\textwidth]{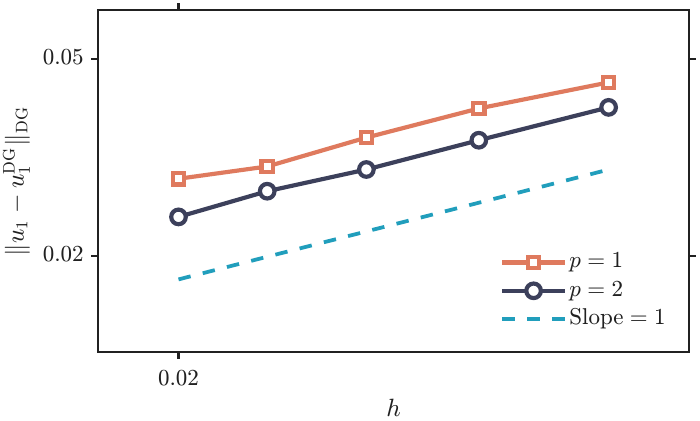}}
  \subfigure[]
  {\includegraphics[width=0.32\textwidth]{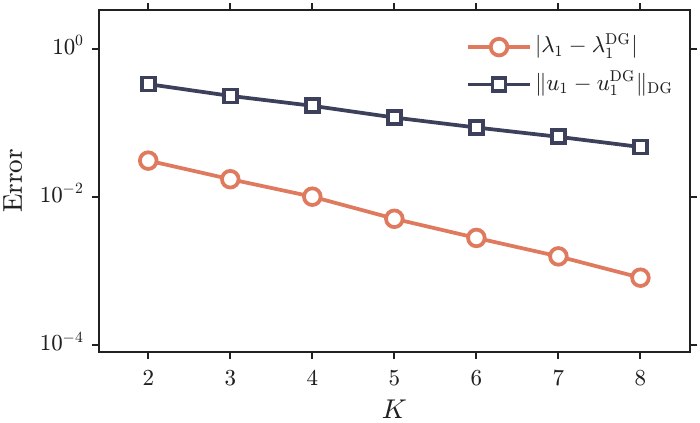}}
  \caption{(Example~3) (a) Convergence of the eigenvalue with respect to $h$ for $K=20$; (b) convergence of the eigenfunction in the DG norm with respect to $h$ for $K=20$; (c) convergence of the eigenvalue and eigenfunction DG error with respect to $K$ for $p=1$ and $h=0.4/2^6$.}
  \label{fig:ex3-convergence}
\end{figure}

\begin{figure}[h!]
\subfigure[Plane waves]{\includegraphics[width=0.32\textwidth]{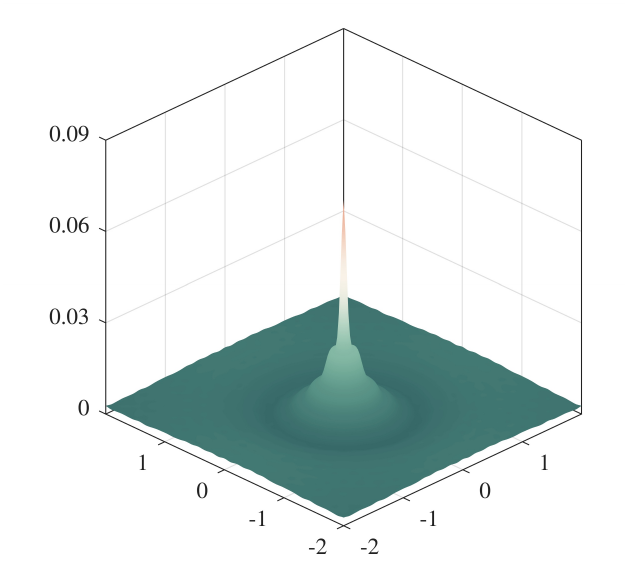}}
\subfigure[IGA]{\includegraphics[width=0.32\textwidth]{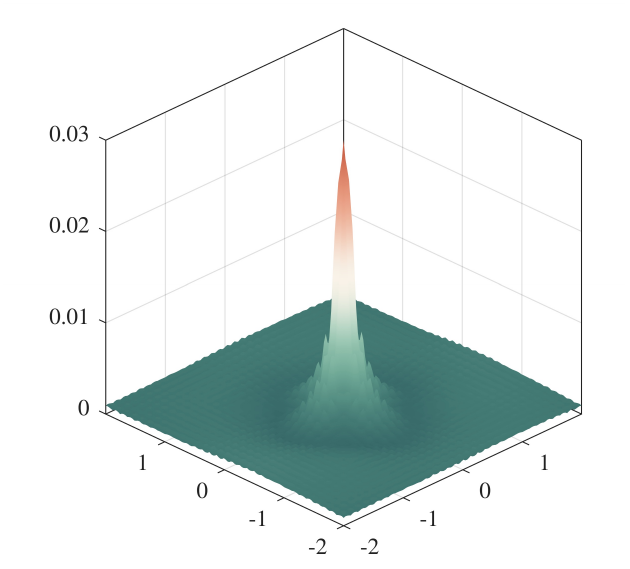}}
\subfigure[IGA-PW]{\includegraphics[width=0.32\textwidth]{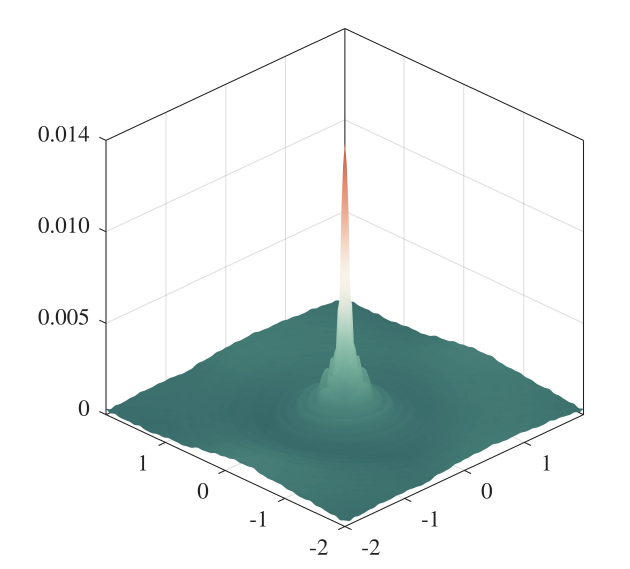}}
\caption{(Example 3) Absolute error in the eigenfunction for the plane wave, IGA, and IGA-PW discretizations, where the numbers of degrees of freedom are
$1257, 1024, 790$, respectively.}
\label{fig:ex3-error-surface}
\end{figure}

\vspace{0.5cm}

\noindent{\bf Example 4 (3D simulation for a periodic helium system).}
As a final three-dimensional nonlinear test, we consider a periodic Hartree model for a helium atom: find $\lambda\in\mathbb{R}$ and $u\in H^1_{\#}(\Omega)$ with $\|u\|_{L_\#^2(\Omega)}=1$ such that
\[
\left(-\frac{1}{2}\Delta + V_{\rm ext} + V_{\rm H} [\rho]  \right)u = \lambda u,
\]
where $\Omega=[-2,2]^3$ and $\rho(\vr)=2|u(\vr)|^2$. The external potential is
\begin{equation*}
V_{\rm ext}(\vr) =2\bigg( - \frac{\operatorname{erfc}(\alpha |\vr|)}{| \vr|}
 - \frac{4 \pi}{|\Omega|}\sum_{\substack{\vG\in\mathcal{R}^*\\ \vG \ne \bm 0}} \frac{1}{|\vG|^2} e^{-|\vG|^2/(4\alpha^2)} e^{\mathrm{i} \vG \cdot \vr}
 + \frac{2\alpha}{\sqrt{\pi}}\bigg).
\end{equation*}
The Hartree potential is denoted by

\begin{equation*}
V_{\rm H}[\rho](\vr)=\int_{\Omega}\frac{\rho(\vr')}{|\vr-\vr'|}\,\dd\vr'.
\end{equation*}
In the simulation, the inner domain is set as $\Omega_{\rm in}=[-0.2,0.2]^3$.

Figure~\ref{fig:ex4-conv} shows the convergence with respect to both \(h\) and the plane-wave cutoff \(K\) for the three-dimensional problem. 
Since \(u_1\in H^{5/2-\varepsilon}(\Omega)\) for any \(\varepsilon>0\) \cite{maday2019regularity}, the observed convergence rates are consistent with Theorem~\ref{them:convergence-rate}.
Furthermore, Figure~\ref{fig:ex4-fields} shows the helium ground state quantities computed by IGA-PW. 
The distributions of $|u|$, $\rho$, and $V_{\rm H}$ on the plane $z=0$ exhibit the expected radial structure, and the radial distributions agree well with the reference solutions. 

\begin{figure}[h!]
\subfigure[]{
\includegraphics[width=0.32\textwidth]{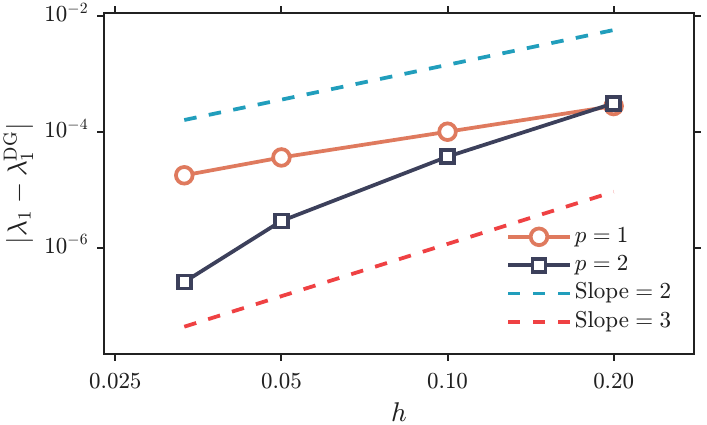}}
\subfigure[]{
\includegraphics[width=0.32\textwidth]{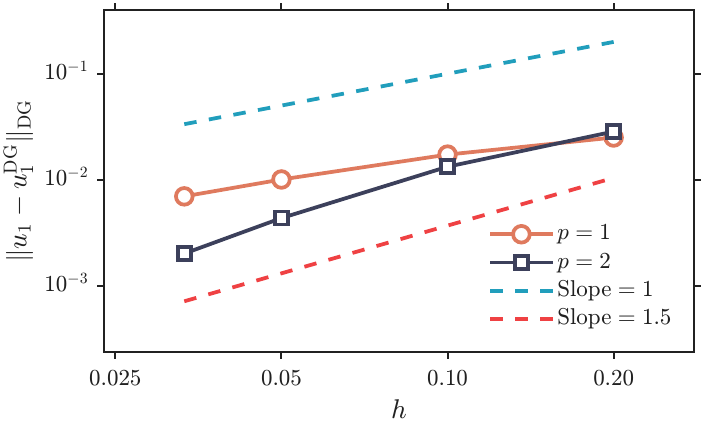}}
\subfigure[]{
\includegraphics[width=0.32\textwidth]{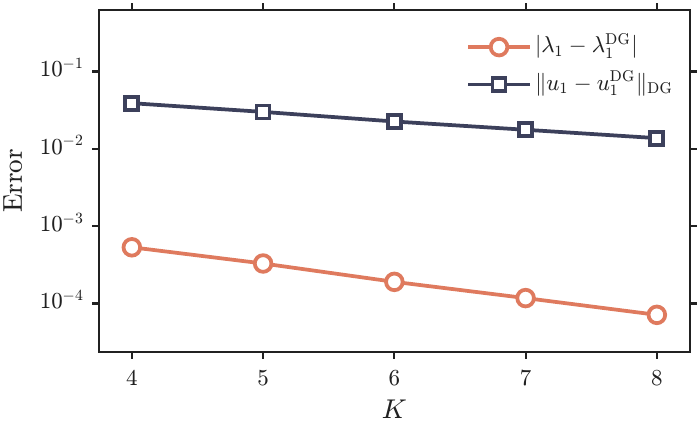}}
\caption{(Example 4) (a) Convergence of the eigenvalue with respect to $h$ for $K=20$; (b) convergence of the eigenfunction in the DG norm with respect to $h$ for $K=20$; (c) convergence of the eigenvalue and eigenfunction DG error with respect to $K$ for $p=1$ and $h=0.4/12$.}
\label{fig:ex4-conv}
\end{figure}

We finally examine the convergence of the ground-state energy. Since the absolute value of the periodic Hartree energy depends on the choice of the Ewald zero-mode convention and the additive constant in \(V_{\rm ext}\), we focus on the convergence toward a reference energy computed under the same convention.
The convergence in Figure~\ref{fig:ex4-energy} shows that the computed ground-state energy approaches the reference value monotonically as the discretization is refined, with the final error reaching nearly \(10^{-4}\). 
This behavior indicates that the IGA-PW discretization accurately captures both the singular near-nucleus structure and the smooth interstitial contribution.

\begin{figure}[h!]
\centering
\setlength{\subfigcapskip}{10pt}

\subfigure[$|u|$, $\rho$, and $V_{\rm H}$ on the plane $z=0$.]{%
\begin{minipage}{0.98\textwidth}
\centering
\includegraphics[height=0.22\linewidth,keepaspectratio]{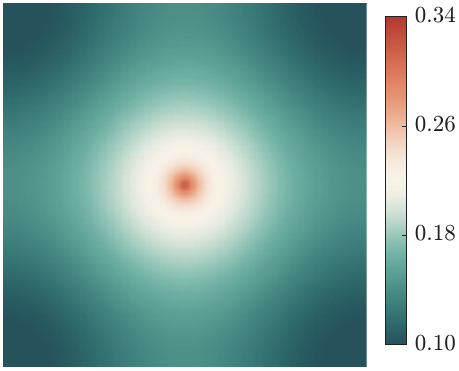}%
\hspace{0.05\linewidth}%
\includegraphics[height=0.22\linewidth,keepaspectratio]{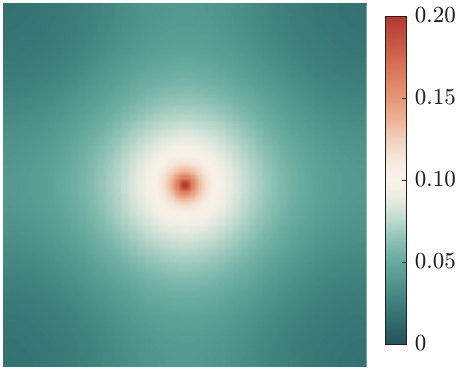}%
\hspace{0.05\linewidth}%
\includegraphics[height=0.22\linewidth,keepaspectratio]{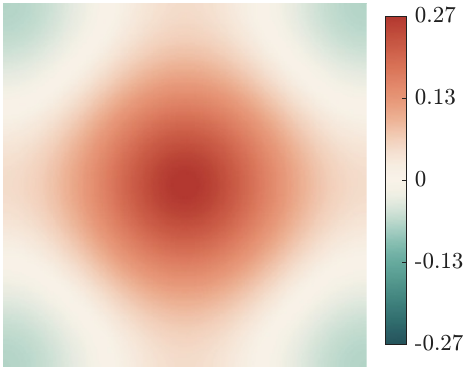}%
\end{minipage}}

\subfigure[Radial distributions compared with reference solutions.]{%
\begin{minipage}{0.98\textwidth}
\centering
\includegraphics[height=0.17\linewidth,keepaspectratio]{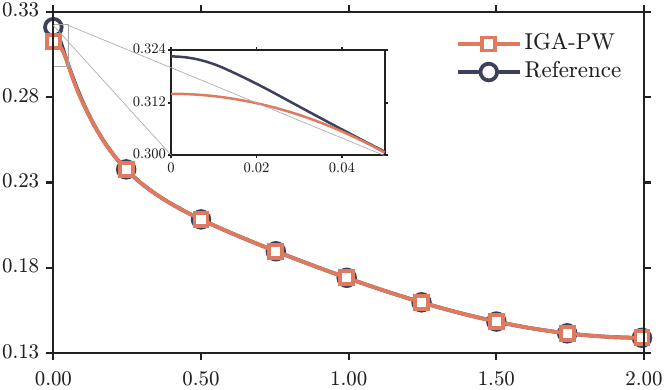}%
\hspace{0.05\linewidth}%
\includegraphics[height=0.17\linewidth,keepaspectratio]{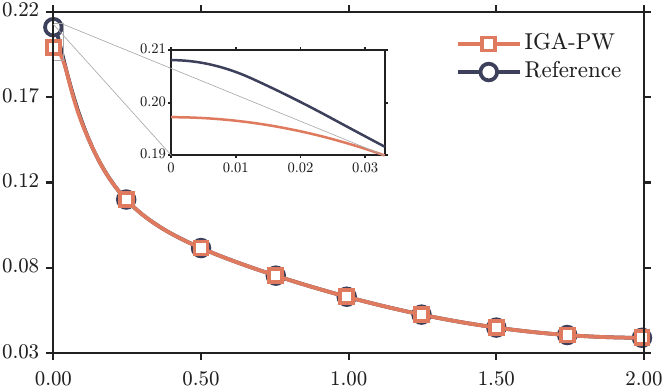}%
\hspace{0.05\linewidth}%
\includegraphics[height=0.17\linewidth,keepaspectratio]{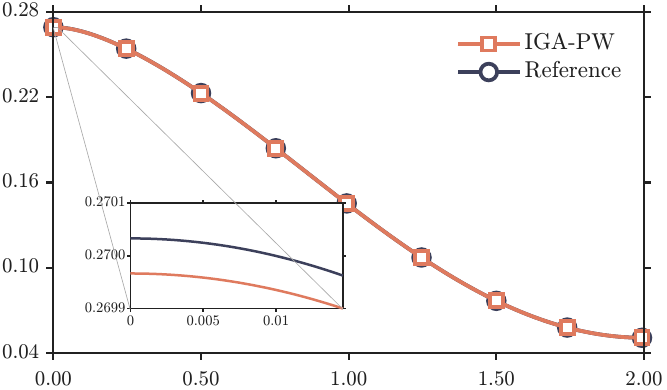}%
\end{minipage}}

\caption{(Example 4) Helium ground-state quantities computed by IGA-PW 
with $K=20$, $p=2$, and $h=0.4/2^3$.}
\label{fig:ex4-fields}
\end{figure}

\begin{figure}[h!]
\centering
\includegraphics[width=0.6\textwidth]{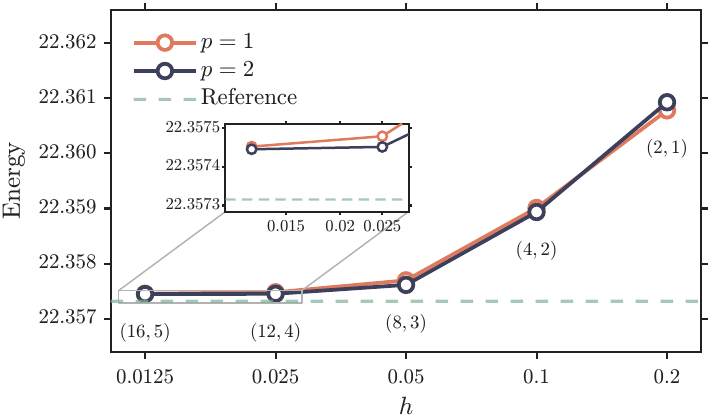}
\caption{Convergence of the shifted periodic energy computed by the IGA-PW method. The labels $(K,r)$ denote the plane wave cutoff $K$ and the mesh size $h=0.4/2^r$, respectively.}
\label{fig:ex4-energy}
\end{figure}

\section{Concluding remarks}
\label{sec:conclusions}

In this paper, we have developed an IGA-PW discretization for periodic full-potential eigenvalue problems arising in electronic structure calculations. 
Tensor-product B-splines are employed in localized atomic patches to resolve the cusp behavior induced by the Coulomb singularity, while plane waves are used in the smooth interstitial region. 
The two approximation spaces are coupled through a symmetric interior penalty DG formulation, providing a flexible weak interface treatment without enforcing strong continuity across patch boundaries.

Compared with global spline discretizations, the proposed hybrid strategy concentrates local resolution near the nuclei, where the reduced regularity is most pronounced, while retaining the efficiency of plane waves in the interstitial region. 
This leads to a more economical use of degrees of freedom while maintaining accurate resolution of the near-nuclear behavior. 
The {\it a priori} error estimates separate the local spline approximation error from the plane wave approximation error in the smooth exterior region, yielding algebraic convergence near the nuclei and superalgebraic convergence in the interstitial region. 
Numerical experiments further demonstrate the accuracy and efficiency of the proposed discretization.

The present DG coupling framework is flexible with respect to the choice of local approximation spaces. 
In particular, the spline discretization in the atomic patches may be replaced by finite element discretizations to allow unstructured local refinement, while the IGA framework can also accommodate more complex local region decompositions and geometries.
Another promising direction is the development of adaptive IGA techniques for Kohn-Sham equations within the present hybrid framework.

\section*{Acknowledgment}

B. Lai and X. Li were partially supported by the National Natural Science Foundation of China (No. 12301548).
X. Meng was partially supported by the National Natural Science Foundation of China (No. 12101057), Guangdong Higher Education Upgrading Plan (UIC-R0400024-21), and Guangdong and Hong Kong Universities ``1+1+1” Joint Research Collaboration Scheme (No. 2025A0505000014).

\appendix
\renewcommand\thesection{\appendixname~\Alph{section}}

\section{Proof of the error estimate}
\label{append:error_analysis}

\renewcommand{\theequation}{A.\arabic{equation}}
\renewcommand{\thelemma}{A.\arabic{lemma}}
\renewcommand{\thetheorem}{A.\arabic{theorem}}
\renewcommand{\theremark}{A.\arabic{remark}}
\renewcommand{\thefigure}{A.\arabic{figure}}

\setcounter{equation}{0}
\setcounter{figure}{0}
\setcounter{lemma}{0}
\setcounter{theorem}{0}
\setcounter{remark}{0}

\makeatletter
\providecommand{\theHequation}{}
\providecommand{\theHfigure}{}
\providecommand{\theHlemma}{}
\providecommand{\theHtheorem}{}
\providecommand{\theHremark}{}
\renewcommand{\theHequation}{appendix.A.\arabic{equation}}
\renewcommand{\theHfigure}{appendix.A.\arabic{figure}}
\renewcommand{\theHlemma}{appendix.A.\arabic{lemma}}
\renewcommand{\theHtheorem}{appendix.A.\arabic{theorem}}
\renewcommand{\theHremark}{appendix.A.\arabic{remark}}
\makeatother

This appendix is devoted to the proof of Theorem~\ref{them:convergence-rate}. 
The analysis follows the DG spectral approximation framework of \cite{li2019dg,li2025dgiga}, but the present IGA-PW setting involves the coupling of two different approximation spaces: local B-splines on the atomic patches and restricted periodic plane waves on the interstitial region.

The proof proceeds in three main steps: first derive the local approximation estimates for the restricted plane waves on \(\Omega_{\rm out}\) and the B-spline discretization on \(\Omega_{\rm in}\); then establish the DG error estimate for the associated source problem within the SIPG framework; and finally apply the abstract DG eigenvalue approximation theory to obtain the eigenfunction and eigenvalue estimates in Theorem~\ref{them:convergence-rate}.

Most parts of the analysis are technical and follow naturally from the frameworks developed in \cite{li2019dg,li2025dgiga}. Therefore, we omit many routine arguments and focus instead on the key intermediate results and the proof ingredients that are specific to the present IGA-PW coupling framework.

We first define the ``best" approximations of the function in the interstitial region and the atomic patch, respectively.
For the interstitial region, we define the projection
\[
P_K:L^2_{\#}(\Omega_{\rm out})\rightarrow X_K(\Omega_{\rm out})
\]
satisfying
\begin{eqnarray*}
\|u-P_K u\|_{H^1(\Omega_{\rm out})}
=
\inf_{U^{\rm out}_K\in X_K(\Omega_{\rm out})}
\|u-U^{\rm out}_K\|_{H^1(\Omega_{\rm out})}.
\end{eqnarray*}
For the atomic patch, we use the B-spline approximation estimate given in \eqref{eq:igaregularity}.

Let $\tilde{\Omega}$ be an inscribed ball of $\Omega_{\rm in}$ (see Figure \ref{fig:inscribed}).
We point out that the introduction of $\tilde{\Omega}$ is solely for numerical analysis and is irrelevant to the implementation of our DG scheme.
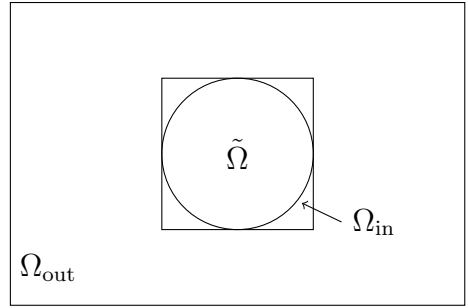
\begin{figure}[ht]
\centering
\begin{tikzpicture}
    \draw[black] (0,0) rectangle (6,4); 
    \draw[black] (2,1) rectangle (4,3); 
    \draw[black] (3,2) circle (1);
    \node at (3,2) {$\tilde{\Omega}$}; 
    \node (inlabel) at (4.8,1.1) {$\Omega_{\text{in}}$};
    \draw[->] (inlabel.west) -- (3.85,1.35);
    \node at (0.5,0.5) {$\Omega_{\text{out}}$}; 
\end{tikzpicture}
\caption{Illustration of the inscribed ball $\tilde{\Omega}$.}
\label{fig:inscribed}
\end{figure}
In order to show the projection error, we present the following smooth extension (SE) condition on functions in \(H^s(\Omega_{\rm out})\):

\vskip 0.2cm 
\noindent
{\bf (SE)}
For \(u\in H^s(\Omega_{\rm out})\), it can be smoothly extended to the domain \(\Omega_{\rm in}\setminus\tilde{\Omega}\), still denoted by \(u\), and there exists a constant \(C\) such that
\begin{equation*}
\|u\|_{H^s(\Omega\setminus\tilde{\Omega})}
\leq C \|u\|_{H^s(\Omega_{\rm out})}.
\end{equation*}

\begin{lemma}[{\bf Restricted plane wave approximation}]
\label{lemma:planewave}
If \(u\in H^s(\Omega_{\rm out})\) and satisfies the condition (SE), then for \(0\le t<s\), there exists a positive constant \(C\) such that
\begin{eqnarray}
\|u-P_K u\|_{H^t(\Omega_{\rm out})}\leq CK^{t-s}\|u\|_{H^s(\Omega_{\rm out})}.
\end{eqnarray}
\end{lemma}

\begin{proof}
It follows from \cite[Proposition~3.1]{li2019dg} and condition~(SE) that
\[
\|u-P_K u\|_{H^t(\Omega_{\rm out})}
\leq
\|u-P_K u\|_{H^t(\Omega\setminus\tilde{\Omega})}
\leq CK^{t-s} \|u\|_{H^s(\Omega\setminus\tilde{\Omega})}
\leq CK^{t-s} \|u\|_{H^s(\Omega_{\rm out})}.
\]
\end{proof}

\begin{remark}[{\bf Smooth extension condition}]
We note that general functions \(u\in H^s(\Omega_{\rm out})\) do not necessarily satisfy the condition (SE), as the interface \(\Gamma\) may have irregular corners.
However, this projection error estimate is ultimately used to derive the convergence rate for the eigenvalue problem, where the eigenfunction is sufficiently smooth near the artificial interface \(\Gamma\). Therefore, the condition (SE) does not affect the validity of our main results.
\end{remark}

We then give an error estimate for the DG
approximation of the corresponding source problem.
Define the solution operators 
\begin{eqnarray*}
T:L_{\#}^2(\Omega)\rightarrow H_{\#}^1(\Omega) \qquad \text{such that} \quad a(Tf,v)=(f,v)\quad\forall~v\in H_{\#}^1(\Omega),
\end{eqnarray*}
and
\begin{eqnarray*}
T^{\rm DG}:L_{\#}^2(\Omega)\rightarrow \mathcal{S}^K_{h}(\Omega)
\qquad \text{such that} \quad a^{\rm DG}(T^{\rm DG}f,v)=(f,v)\quad\forall~v\in \mathcal{S}^K_{h}(\Omega).
\end{eqnarray*}
The following lemma gives the error bound in the DG norm.

\begin{lemma}[{\bf Error estimate in DG norm}]
\label{thm:T-approximate}
Let \(C_{\sigma}\) be sufficiently large and $h=O(K^{-1})$. 
If \(Tf\in H^{s_{\rm in}}(\Omega_{\rm in})\oplus H^{s_{\rm out}}(\Omega_{\rm out})\) and \(Tf|_{\Omega_{\rm out}}\) satisfies the condition (SE)
for \(f\in L^2_{\#}(\Omega)\),
then there exists a positive constant \(C\), independent of \(h\) and \(K\), such that
\begin{eqnarray}
\label{rate-T-tildeT}
\|(T-T^{\rm DG})f\|_{\rm DG}
\leq C h^{k-1} \|Tf\|_{H^{k}(\Omega_{\rm in})} + C \left( K^{1-s_{\rm out}} + \frac{K^{1/2-s_{\rm out}}}{\sqrt{h}} \right) \|Tf\|_{H^{s_{\rm out}}(\Omega_{\rm out})}
\end{eqnarray}
with \(k=\min\{p+1,s_{\rm in}\}\).
\end{lemma}

Using standard duality techniques, we can also derive the following lower-norm estimate.

\begin{lemma}[{\bf Error estimate in \(L^2\) norm}]
\label{thm:error_L2}
Under the conditions in Lemma \ref{thm:T-approximate} and if $h\sim K^{-\alpha}$ with some $1\leq\alpha<3$, there exists a positive constant $C$ such that
\begin{eqnarray}
\|(T-T^{\rm DG})f\|_{L^2(\Omega)} \leq C h^{\frac{3-\alpha}{2\alpha}} \|(T-T^{\rm DG})f\|_{\rm DG}.
\end{eqnarray}
\end{lemma}

Note that the eigenvalue problems \eqref{model-eq} and \eqref{eq-eigen-DG} are equivalent to \(\lambda Tu=u\)
and \(\lambda^{\rm DG}T^{\rm DG}u^{\rm DG}=u^{\rm DG}\), respectively.
The following lemma relates the error of the eigenvalue problem to that of the corresponding source problem.

\begin{lemma}[{\bf Convergence of eigenvalue problem}]
\label{thm:eigen-source-bound}
Let nonzero $\lambda_i$ be the $i$-th eigenvalue of \eqref{model-eq} and $(\lambda_i^{\rm DG}, u_i^{\rm DG})$ be the $i$-th eigenpair of \eqref{eq-eigen-DG}. 
Then $\lambda_i^{\rm DG}\rightarrow \lambda_i$ as $K\rightarrow +\infty$ and $h=O(K^{-1})$, and there exists $u_i\in M(\lambda_i)$ with $\|u_i\|_{L_{\#}^2(\Omega)}=1$ such that  
\begin{align}
\label{source-eigenvalue-a}
&\| u_i^{\rm DG}-u_i \|_{L^2(\Omega)} \leq C \| (T-T^{\rm DG})u_i \|_{L^2(\Omega)},  \\[1ex]\label{source-eigenvalue-b}
 &|\lambda^{\rm DG}_i - \lambda_i| \leq C \|(T-T^{\rm DG})u_i\|_{\rm DG} \inf_{v\in \mathcal{S}^K_{h}(\Omega)} \|u_i-v\|_{\rm DG} + C \| (T-T^{\rm DG})u_i \|^2_{L^2(\Omega)}, \\[1ex]\label{source-eigenvalue-c}
 &\| u_i^{\rm DG}-u_i \|_{\rm DG} \leq C \| (T-T^{\rm DG})u_i \|_{\rm DG} + C \| (T-T^{\rm DG})u_i \|_{L^2(\Omega)}.
\end{align}
\end{lemma}

Combining Lemmas~\ref{thm:T-approximate}, \ref{thm:error_L2}, and \ref{thm:eigen-source-bound} with the regularity of the eigenspace in \eqref{eq:eigenspace}, we can obtain the final error estimates in Theorem~\ref{them:convergence-rate}.

\bibliographystyle{plain}
\bibliography{ref}

@article{castillo2002performance,
  title={{Performance of discontinuous Galerkin methods for elliptic PDEs}},
  author={Castillo, P.},
  journal={SIAM J. Sci. Comput.},
  volume={24},
  number={2},
  pages={524--547},
  year={2002},
  publisher={SIAM}
}

@article{wathen2015preconditioning,
  author  = {Wathen, A.J.},
  title   = {Preconditioning},
  journal = {Acta Numer.},
  volume  = {24},
  year    = {2015},
  pages   = {329--376},
  doi     = {10.1017/S0962492915000021}
}

@article{holzmann2005optimized,
  author  = {Holzmann, M. and Bernu, B.},
  title   = {Optimized periodic \(1/r\) {Coulomb} potential in two dimensions},
  journal = {J. Comput. Phys.},
  volume  = {206},
  number  = {1},
  year    = {2005},
  pages   = {111--121},
  doi     = {10.1016/j.jcp.2004.11.037}
}

@article{natoli1995optimized,
  author  = {Natoli, V. and Ceperley, D.M.},
  title   = {An optimized method for treating long-range potentials},
  journal = {J. Comput. Phys.},
  volume  = {117},
  number  = {1},
  year    = {1995},
  pages   = {171--178},
  doi     = {10.1006/jcph.1995.1054}
}

@article{arnold2002unified,
  author  = {Arnold, D.N. and Brezzi, F. and Cockburn, B. and Marini, L.D.},
  title   = {Unified analysis of discontinuous {Galerkin} methods for elliptic problems},
  journal = {SIAM J. Numer. Anal.},
  volume  = {39},
  number  = {5},
  year    = {2002},
  pages   = {1749--1779},
  doi     = {10.1137/S0036142901384162}
}

@book{martin05,
  author={R.M. Martin},
  title={Electronic Structure: Basic Theory and Practical Methods},
  publisher={Cambridge University Press},
  year={2005},
}

@article{lin2019numerical,
  title={{Numerical methods for {K}ohn-{S}ham density functional theory}},
  author={Lin, L. and Lu, J. and Ying, L.},
  journal={Acta Numer.},
  volume={28},
  pages={405--539},
  year={2019},
  publisher={Cambridge University Press}
}

@article{DG_IGA_Eigenvalue_Codes,
  title={{https://github.com/baoweilai/IGA-PW}},
  author={IGA-PW},
  year={2026}
}

@article{hohenberg1964inhomogeneous,
  title={{Inhomogeneous electron gas}},
  author={Hohenberg, P. and Kohn, W.},
  journal={Phys. Rev.},
  volume={136},
  number={3B},
  pages={B864--B871},
  year={1964},
  publisher={American Physical Society}
}

@article{saad2010numerical,
  title={{Numerical methods for electronic structure calculations of materials}},
  author={Saad, Y. and Chelikowsky, J.R. and Shontz, S.M.},
  journal={SIAM Rev.},
  volume={52},
  number={1},
  pages={3--54},
  year={2010},
  publisher={SIAM}
}

@article{PRIMME2017algorithm,
  title={{PRIMME\_SVDS: A high-performance preconditioned SVD solver for accurate large-scale computations}},
  author={L. Wu and E. Romero and A. Stathopoulos},
  journal={SIAM J. Sci. Comput.},
  volume={39},
  number={5},
  pages={248-271},
  year={2017},
}

@article{kohn1965self,
  title={Self-consistent equations including exchange and correlation effects},
  author={Kohn, W. and Sham, L.J.},
  journal={Phys. Rev.},
  volume={140},
  number={4A},
  pages={A1133},
  year={1965},
  publisher={APS}
}

@misc{pitaevskii2003bose,
  title={{Bose-Einstein Condensation, Clarendon}},
  author={Pitaevskii, L.P. and Stringari, S.},
  year={2003},
  publisher={Oxford}
}

@article{duvigneau2018isogeometric,
  title={{Isogeometric analysis for compressible flows using a discontinuous Galerkin method}},
  author={Duvigneau, R.},
  journal={Comput. Methods Appl. Mech. Engrg.},
  volume={333},
  pages={443--461},
  year={2018},
  publisher={Elsevier}
}

@article{moore2019space,
  title={{Space-time multipatch discontinuous Galerkin isogeometric analysis for parabolic evolution problems}},
  author={Moore, S.E.},
  journal={SIAM J. Numer. Anal.},
  volume={57},
  number={3},
  pages={1471--1493},
  year={2019},
  publisher={SIAM}
}

@article{cimrman2018isogeometric,
  title={Isogeometric analysis in electronic structure calculations},
  author={Cimrman, R. and Nov{\'a}k, M. and Kolman, R. and T\r{u}ma, M. and Vack{\'a}{\v{r}}, J.},
  journal={Math. Comput. Simul.},
  volume={145},
  pages={125--135},
  year={2018},
  publisher={Elsevier}
}

@article{cimrman2018convergence,
  title={{Convergence study of isogeometric analysis based on B{\'e}zier extraction in electronic structure calculations}},
  author={Cimrman, R. and Nov{\'a}k, M. and Kolman, R. and T\r{u}ma, M. and Ple{\v{s}}ek, J. and Vack{\'a}{\v{r}}, J.},
  journal={Appl. Math. Comput.},
  volume={319},
  pages={138--152},
  year={2018},
  publisher={Elsevier}
}

@article{hughes2005isogeometric,
  title={{Isogeometric analysis: CAD, finite elements, NURBS, exact geometry and mesh refinement}},
  author={Hughes, T.J.R. and Cottrell, J.A. and Bazilevs, Y.},
  journal={Comput. Methods Appl. Mech. Engrg.},
  volume={194},
  number={39-41},
  pages={4135--4195},
  year={2005},
  publisher={Elsevier}
}

@article{pask2001finite,
  title={Finite-element methods in electronic-structure theory},
  author={Pask, J.E. and Klein, B.M. and Sterne, P.A. and Fong, C.Y.},
  journal={Comput. Phys. Commun.},
  volume={135},
  number={1},
  pages={1--34},
  year={2001},
  publisher={Elsevier}
}

@article{batcho2000computational,
  title={Computational method for general multicenter electronic structure calculations},
  author={Batcho, P.F.},
  journal={Phys. Rev. E},
  volume={61},
  number={6},
  pages={7169},
  year={2000},
  publisher={APS}
}

@article{proserpio2020framework,
  author  = {Proserpio, D. and Ambati, M. and De Lorenzis, L. and Kiendl, J.},
  title   = {A framework for efficient isogeometric computations of phase-field brittle fracture in multipatch shell structures},
  journal = {Comput. Methods Appl. Mech. Engrg.},
  volume  = {372},
  pages   = {113363},
  year    = {2020},
  doi     = {10.1016/j.cma.2020.113363}
}

@article{CHAN201822,
title = {{Multi-patch discontinuous Galerkin isogeometric analysis for wave propagation: Explicit time-stepping and efficient mass matrix inversion}},
journal = {Comput. Methods Appl. Mech. Engrg.},
volume = {333},
pages = {22--54},
year = {2018},
issn = {0045-7825},
doi = {https://doi.org/10.1016/j.cma.2018.01.022},
url = {https://www.sciencedirect.com/science/article/pii/S0045782518300240},
author = {J. Chan and J.A. Evans}
}

@article{da2014mathematical,
  title={Mathematical analysis of variational isogeometric methods},
  author={ Beir\~{a}o da Veiga, L. and Buffa, A. and Sangalli, G. and V{\'a}zquez, R.},
  journal={Acta Numer.},
  volume={23},
  pages={157--287},
  year={2014},
  publisher={Cambridge University Press}
}

@book{schumaker2007spline,
  title={{Spline Functions: Basic Theory}},
  author={Schumaker, L.L.},
  year={2007},
  publisher={{Cambridge University Press}}
}

@article{maday2019regularity,
  title={{Regularity and $hp$ discontinuous Galerkin finite element approximation of linear elliptic eigenvalue problems with singular potentials}},
  author={Maday, Y. and Marcati, C.},
  journal={Math. Models Methods Appl. Sci.},
  volume={29},
  number={08},
  pages={1585--1617},
  year={2019},
  publisher={World Scientific}
}

@article{andersen75,
  title={Linear methods in band theory},
  author={Andersen, O.K.},
  journal={Phys. Rev. B},
  volume={12},
  number={8},
  pages={3060--3083},
  year={1975},
}

@article{antonietti06,
	title={Discontinuous {G}alerkin approximation of the {L}aplace eigenproblem},
	author={Antonietti, P.F. and Buffa, A. and Perugia, I.},
	journal={Comput. Methods Appl. Mech. Engrg.},
	volume={195},
	number={25},
	pages={3483--3503},
	year={2006},
}

@article{arnold82,
	title={An Interior Penalty Finite Element Method with Discontinuous Elements},
	author={Arnold, D.N.},
	journal={SIAM J. Numer. Anal.},
	volume={19},
	number={4},
	pages={742--760},
	year={1982},
}

@article{chen13,
	title={Numerical analysis of finite dimensional approximations of {K}ohn–{S}ham models},
	author={Chen, H. and Gong, X. and He, L. and Yang, Z. and A. Zhou},
	journal={Adv. Comput. Math.},
	volume={38},
	number={2},
	pages={225--256},
	year={2013},
}

@article{flad08,
	title={Asymptotic regularity of solutions to {H}artree-{F}ock equations with {C}oulomb potential},
	author={Flad, H.J. and Schneider, R. and Schulze, B.W.},
	journal={Math. Meth. Appl. Sci.},
	volume={31},
	number={18},
	pages={2172--2201},
	year={2008},
}

@article{fournais04,
	title={Analyticity of the density of electronic wavefunctions},
	author={Fournais, S. and Hoffmann-Ostenhof, M. and Hoffmann-Ostenhof, T. and \O{}stergaard S\o{}rensen, T.},
	journal={Ark. Mat.},
	volume={42},
	number={1},
	pages={87--106},
	year={2004},
}

@article{fournais07,
	title={Non-Isotropic Cusp Conditions and Regularity of the Electron Density of Molecules at the Nuclei},
	author={Fournais, S. and Hoffmann-Ostenhof, M. and Hoffmann-Ostenhof, T. and \O{}stergaard S\o{}rensen, T.},
	journal={Ann. Henri Poincar\'{e}},
	volume={8},
	number={4},
	pages={731--748},
	year={2007},
}

@article{herring40,
	title={A New Method for Calculating Wave Functions in Crystals},
	author={Herring, C.},
	journal={Phys. Rev.},
	volume={57},
	number={12},
	pages={1169--1177},
	year={1940},
}

@article{sjostedt00,
	title={An alternative way of linearizing the augmented plane-wave method},
	author={E. Sj\"{o}stedt and L. Nordstr\"{o}m and D.J. Singh},
	journal={Solid State Commun.},
	volume={114},
	number={1},
	pages={15--20},
	year={2000},
}

@article{slater37,
	title={Wave functions in a periodic potential},
	author={J.C. Slater},
	journal={Phys. Rev.},
	volume={51},
	number={10},
	pages={846--851},
	year={1937},
}

@incollection{cances00,
	author={Canc\`{e}s, E.},
	title={{SCF} algorithms for {HF} electronic calculations},
	bookTitle={Mathematical Models and Methods for Ab Initio Quantum Chemistry},
	editor={M. Defranceschi and C. Le Bris},
	year={2000},
	publisher={Springer Berlin Heidelberg},
	pages={17--43},
	volume={74},
}

@article{motamarri2013higher,
  author  = {Motamarri, P. and Nowak, M.R. and Leiter, K. and Knap, J. and Gavini, V.},
  title   = {Higher-order adaptive finite-element methods for {Kohn--Sham} density functional theory},
  journal = {J. Comput. Phys.},
  volume  = {253},
  year    = {2013},
  pages   = {308--343}
}

@article{bao2012adaptive,
  author  = {Bao, G. and Hu, G. and Liu, D.},
  title   = {An {$h$}-adaptive finite element solver for the calculations of the electronic structures},
  journal = {J. Comput. Phys.},
  volume  = {231},
  number  = {14},
  year    = {2012},
  pages   = {4967--4979}
}

@article{chen2014adaptive,
  author  = {Chen, H. and Dai, X. and Gong, X. and He, L. and Zhou, A.},
  title   = {Adaptive finite element approximations for {Kohn--Sham} models},
  journal = {Multiscale Model. Simul.},
  volume  = {12},
  number  = {4},
  year    = {2014},
  pages   = {1828--1869}
}

@article{masud2012bsplines,
  author  = {Masud, A. and Kannan, R.},
  title   = {{B}-splines and {NURBS} based finite element methods for {Kohn--Sham} equations},
  journal = {Comput. Methods Appl. Mech. Engrg.},
  volume  = {241--244},
  year    = {2012},
  pages   = {112--127}
}

@article{cances2016optimal,
  author  = {Canc{\`e}s, E. and Mourad, N.},
  title   = {Existence of a type of optimal norm-conserving pseudopotentials for {Kohn--Sham} models},
  journal = {Commun. Math. Sci.},
  volume  = {14},
  number  = {5},
  year    = {2016},
  pages   = {1315--1352}
}

@article{cances2012planewave,
  author  = {Canc{\`e}s, E. and Chakir, R. and Maday, Y.},
  title   = {Numerical analysis of the planewave discretization of some orbital-free and {Kohn--Sham} models},
  journal = {ESAIM Math. Model. Numer. Anal.},
  volume  = {46},
  number  = {2},
  year    = {2012},
  pages   = {341--388},
  doi     = {10.1051/m2an/2011038}
}

@article{fournais2002smooth,
  author  = {Fournais, S. and Hoffmann-Ostenhof, M. and Hoffmann-Ostenhof, T. and \O{}stergaard S\o{}rensen, T.},
  title   = {The electron density is smooth away from the nuclei},
  journal = {Commun. Math. Phys.},
  volume  = {228},
  number  = {3},
  year    = {2002},
  pages   = {401--415},
  doi     = {10.1007/s002200200668}
}

@article{chen2015apw,
  author  = {Chen, H. and Schneider, R.},
  title   = {Numerical analysis of augmented plane wave methods for full-potential electronic structure calculations},
  journal = {ESAIM Math. Model. Numer. Anal.},
  volume  = {49},
  number  = {3},
  year    = {2015},
  pages   = {755--785},
  doi     = {10.1051/m2an/2014052}
}

@article{wang2025hierarchical,
  author  = {Wang, T. and Kuang, Y. and Zhang, R. and Hu, G.},
  title   = {A hierarchical splines-based {$h$}-adaptive isogeometric solver for all-electron {Kohn--Sham} equation},
  journal = {J. Comput. Phys.},
  volume  = {534},
  year    = {2025},
  pages   = {114003}
}

@article{lin2012adaptive,
  author  = {Lin, L. and Lu, J. and Ying, L. and E, W.},
  title   = {Adaptive local basis set for {Kohn--Sham} density functional theory in a discontinuous {Galerkin} framework {I}: Total energy calculation},
  journal = {J. Comput. Phys.},
  volume  = {231},
  number  = {4},
  year    = {2012},
  pages   = {2140--2154}
}

@article{zhang2017adaptive,
  author  = {Zhang, G. and Lin, L. and Hu, W. and Yang, C. and Pask, J. E.},
  title   = {Adaptive local basis set for {Kohn--Sham} density functional theory in a discontinuous {Galerkin} framework {II}: Force, vibration, and molecular dynamics calculations},
  journal = {J. Comput. Phys.},
  volume  = {335},
  year    = {2017},
  pages   = {426--443}
}

@article{li2019dg,
  author  = {Li, X. and Chen, H.},
  title   = {A discontinuous {Galerkin} scheme for full-potential electronic structure calculations},
  journal = {J. Comput. Phys.},
  volume  = {385},
  year    = {2019},
  pages   = {33--50}
}

@article{li2025dgiga,
  author  = {Li, X. and Meng, X.},
  title   = {Numerical analysis of multi-patch discontinuous {Galerkin} isogeometric method for full-potential electronic structure calculations},
  journal = {J. Comput. Appl. Math.},
  volume  = {477},
  year    = {2025},
  pages   = {117125}
}

@book{lebris03,
	author={Le Bris, C.},
	title={Handbook of Numerical Analysis},
	volume={X, special issue: Computational Chemistry},
	publisher={North-Holland},
	year={2003},
}

\end{document}